%% file: main.tex
\documentclass[11pt]{amsart}
\usepackage[utf8]{inputenc}
\usepackage{amsmath,amsthm,amscd,amssymb}
\usepackage[margin=1in]{geometry}
\usepackage{color}
\usepackage{fancyhdr}
\usepackage{amsfonts}
\usepackage{mathtools}
\usepackage{etoolbox}
\usepackage{amsfonts}
\usepackage{cases}
\usepackage{tikz-cd}
\usepackage{faktor} 
\usepackage[backend=biber,style=alphabetic,sorting=nty, url=false,doi=false,isbn=false]{biblatex}
\usepackage[colorlinks=true, linkcolor=blue, citecolor=blue, urlcolor=blue]{hyperref}

\usepackage{esint}
\numberwithin{equation}{section}
\newtheorem{Theorem}{Theorem}[section]
\newtheorem{Proposition}[Theorem]{Proposition}
\newtheorem{Lemma}[Theorem]{Lemma}
\newtheorem{Corollary}[Theorem]{Corollary}
\newtheorem{Claim}[Theorem]{Claim}
\theoremstyle{definition}
\newtheorem{Definition}[Theorem]{Definition}

\newtheorem{Remark}[Theorem]{Remark}

\title{An $\epsilon$-Regularity Theorem for Non-collapsed Ricci Flow}

\author[Fluck]{Harry Fluck}
\address{Department of Mathematics, Cornell University \\ Ithaca, NY 14853, USA}
\email{hpf5@cornell.edu}

\author[Hallgren]{Max Hallgren}
\address{Department of Mathematics, Rutgers University \\ New Brunswick, NJ, 08854, USA}
\email{mh1564@scarletmail.rutgers.edu}

\begin{document}

\maketitle

\allowdisplaybreaks

\begin{abstract}
    In this article we prove an $\epsilon$-regularity theorem for non-collapsed Ricci flows which are K\"ahler or have dimension four. Using this, we obtain volume estimates for high-curvature regions of such Ricci flows under a finite energy condition. As an application, we recover known estimates for singularity models of Fano K\"ahler-Ricci flows. 
    In the course of our proof, we find a criterion for uniform convergence of solutions to the heat equation along a sequence of $\mathbb{F}$-converging Ricci flows, and apply this to new parabolic regularizations of some natural geometric quantities.
\end{abstract}

\input{Introduction}

\section*{Acknowledgments}

The authors thank Xiaodong Cao and Richard Bamler for useful comments and feedback. The second author is supported in part by the National Science Foundation under Grant No.~DMS-2202980. 

\input{Preliminaries}
\input{Strong_Potentials}

\input{Alt_Tech_Lemma}

\input{Epsilon_Regularity}

\input{Minkowski2}

\input{FanoKRF}

\printbibliography

\end{document}

%% file: Introduction.tex
\section{Introduction}

The Ricci flow was first introduced in \cite{3manPRc} by R. S. Hamilton, and has since developed into a rich field with many applications in both geometry and topology \cite{PCO,QuarterPinSphere, 4PIC}, including the resolution of Thurston's geometrization conjecture by Perelman \cite{perelman}. In many applications, a key step is to understand precisely the kinds of finite time singularities that may occur. Often, this is done via a blowup procedure, in which one takes limits of rescalings of the flow, producing so-called singularity models. In dimension three and lower, such singularity models are smooth, and have been classified \cite{RF_surfaces,RF_surfaces_2, brendle2020,Brendle_2021}. In higher dimensions, a compactness theory was recently developed \cite{Bam1,Bam2,Bam3} which guarantees the existence of singularity models in great generality, though they may be singular. In \cite{Bam3} it was shown that the singular set of a non-collapsed limit of Ricci flows is of parabolic Minkowski codimension $4$. However, it remains an important open question whether this may be improved to uniform codimension $4$ Minkowski content estimates for the almost singular set. In the setting of Riemannian manifolds with bounded Ricci curvature, the analogous uniform estimate was shown to hold in \cite{L2CurvEin}. Prior to that article, it was shown in \cite{ChNaber2} using an $\epsilon$-regularity theorem \cite{CCT} together with a covering argument, that this conjecture is true in the K\"ahler case, assuming some $L^2$-bounds for the curvature tensor. 

The main goal of this article is to extend the aforementioned results of \cite{CCT} and \cite{ChNaber2} to the Ricci flow setting. To this end, we first prove the following $\epsilon$-regularity theorem in the Ricci flow setting, which is analogous to that of \cite[Theorem 1.11]{CCT}. All relevant definitions will be reviewed in Section \ref{section:preliminaries}.

\begin{Theorem} \label{thm:epsreg}
For any $Y<\infty$, there exists $\epsilon_0=\epsilon_0(Y)>0$ such that the following holds. Let $(M^n,(g_t)_{t\in I})$ be a closed K\"ahler Ricci flow of (real) dimension $n$ and $(x_0,t_0)\in M\times I$, $r>0$ satisfy $\mathcal{N}_{x_0,t_0}(r^2)\geq -Y$. 
If $(x_0,t_0)$ is $(n-4,\epsilon_0,r)$-split and
\begin{equation} \label{eq:smallcurvature} r^{2-n}\int_{t_0-2r^2}^{t_0-\epsilon_0 r^2}\int_{P^\ast_t(x_0,t_0; r)}|Rm|_{g_t}^2\,dg_tdt<\epsilon_0,\end{equation}
then $r_{\operatorname{Rm}}(x_0,t_0)\geq \epsilon_0 r$. 
If $n=4$, this also holds without the K\"ahler assumption.
\end{Theorem}

\begin{Remark} The lower bound $\mathcal{N}_{x_0,t_0}(r^2)\geq -Y$ on Nash entropy (defined by \eqref{eq:nashentropy}) is the appropriate non-collapsing assumption at scale $r$.

Roughly speaking, $P_t^{\ast}(x_0,t_0;r)$ is a better behaved version of the geodesic ball $B(x_0,t_0,r)$ for Ricci flows (see Definition \ref{def:Pstarparabolic}). A point in a Ricci flow is $(n-4,\epsilon,r)$-split if the flow near this point is $\epsilon$-close at scale $r$ to isometrically splitting a factor of $\mathbb{R}^{n-4}$ (see Definition \ref{def : strong_splitting} for a precise definition).  This condition is natural for estimating the size of the singular set as it is satisfied by a quantifiably large set of points in the almost-singular set, at many different scales (c.f. \cite[Proposition 2.25]{Bam3} or Lemma \ref{lem:quant_strat}). The same is true for the assumption \eqref{eq:smallcurvature} near points which satisfy the following finite energy condition.
\end{Remark}

\begin{Remark} \label{rem:riemannian} If $n>4$ and the K\"ahler assumption is dropped, then the following modification of Theorem \ref{thm:epsreg} is true. For any $\sigma>0$, there exists $\epsilon_0=\epsilon_0(Y,\sigma)$ such that for any $\epsilon \in (0,\epsilon_0]$, if $\mathcal{N}_{x_0,t_0}(r^2)\geq -Y$, $(x_0,t_0)$ is $(n-4,\epsilon,r)$-split, and \eqref{eq:smallcurvature} holds, then either $r_{\operatorname{Rm}}(x_0,t_0)\geq \epsilon_0 r$ or 
    \begin{equation} \label{eq:exceptional} d_{\mathbb{F}}((M^n,(g_t)_{t\in [t_0-\sigma^{-1}r^2,t_0-\sigma r^2]},(\nu_{x_0,t_0;t})_{t\in [t_0-\sigma^{-1}r^2,t_0-\sigma r^2]}, (\mathcal{X}, (\nu_{x_\infty;t})_{t\in [t_0-\sigma^{-1}r^2,t_0-\sigma r^2]})<\sigma r,\end{equation}  
where $\mathcal{X}$ is a static metric flow modeled on the flat cone $\mathbb{R}^{n-4}\times C(L_{p,q})$, where $L_{p,q}$ is an exceptional Lens space with $ q^{2}\equiv -1 \mod p$.
We expect that \eqref{eq:exceptional} does not occur given these assumptions, as this is true in the setting of bounded Ricci curvature (see \cite[Theorem 1]{ChenDonaldson}), but this cannot be shown using our techniques.
\end{Remark}

\begin{Definition} \label{def:finite_energy}
Fix $A,D<\infty $ and $r,\zeta>0$. Let  $(M^{n},(g_t)_{t\in I})$ be a closed Ricci flow and $(x_0,t_0)\in M\times I$. We say that $(x_{0},t_{0})$ satisfies the $(A,\zeta,D,r)$-finite energy condition if
$$r^{2-n}\int_{\max(t_0-r^2A^2,\inf(I))}^{t_0-r^2\zeta }\int_{P^{\ast}_t(x_{0},t_{0};Ar) } |Rm|_{g_{t}}^2 dg_{t}dt \leq D.$$

Given a function $\Phi: [1,\infty) \times (0,1] \to (1,\infty)$, a sequence of Ricci flows $(M^{n}_{i},(g_{i,t})_{t\in [-T_i,0]},(\nu_{x_{i},0;t})_{t\in [-T_i,0]})$  satisfies the $\Phi$-finite energy condition if for each $A<\infty$ and $\zeta>0$, there exists $i(A,\zeta)\in \mathbb{N}$ such that  $(x_i,0)$ satisfies the  $(A,\zeta,\Phi(A,\zeta),1)$-finite energy condition  for all $i\geq i(A,\zeta)$. 
\end{Definition}

\begin{Remark}
  Note that the $\Phi$-finite energy condition holds along blowup sequences of a fixed Fano K\"ahler Ricci flow for some $\Phi$ depending on the given flow (c.f. the proof of Theorem \ref{thm:Kahler}). We conjecture that for \textit{any} sequence of Ricci flows $(M^{n}_{i},(g_{i,t})_{t\in [-T_i,0]},(\nu_{x_{i},0;t})_{t\in [-T_i,0]})$ with $\mathcal{N}_{x_i,0}(1)\geq -Y$, the $\Phi$-finite energy assumption holds \textit{apriori} for some function $\Phi$ that depends only on $Y$. The corresponding statement with the $L^2$ bounds in Definition \ref{def:finite_energy} replaced by $L^p$ bounds where $p<2$ is known \cite[Theorem 2.28]{Bam3}. 
\end{Remark}

\begin{Theorem} \label{thm:smoothminkowski}
    For any $A,Y<\infty$, there exists $C(A,Y)<\infty$ such that the following statement holds. Let $(M^{n},g(t)_{t\in I})$ be a closed K\"ahler Ricci flow and $(x_{0},t_{0})\in M \times I$, $r>0$ satisfy  $\mathcal{N}_{x_{0},t_{0}}(r^{2})\geq -Y$. If $(x_0,t_0)$ satisfies the  $(A+2,\zeta,D,r)$-finite energy condition for some $D<\infty$ and $\zeta>0$, and if $[t_0 - 2r^2 A^2,t_0] \subseteq I$, then
       \begin{align*}
        \int_{t_0-r^2 A^2}^{t_0-r^2 \zeta}|\{x\in P_t^{\ast}(x_0,t_0;rA) \,\arrowvert \, r_{\operatorname{Rm}}(x,t)<r\sigma \}|_{g_t} dt\leq CDr^{n+2}\sigma ^{4}.
    \end{align*}
\end{Theorem}

Passing to the limit easily yields the following. We refer the reader to \cite[paragraph following Theorem 2.29]{Bam3} for the precise definition of $r_{\operatorname{Rm}}'$, which generalizes the curvature scale to metric flows. 

\begin{Corollary} \label{cor : limitminkowski}
    For any $A,Y <\infty$, and any function $\Phi$ as in Definition \ref{def:finite_energy}, there exists $C=C(A,Y,\Phi)<\infty$ such that the following holds. Let $(\mathcal{X},(\nu_{x;s})_{s\in [-T,0]})$ be a metric flow arising as an $\mathbb{F}$-limit of closed K\"ahler Ricci flows $(M^{n}_{i},(g_{i,t})_{t\in [-T_{i},0]},(\nu_{x_{i},0;t})_{t\in [-T_i,0]})$ that satisfy the $\Phi$-finite energy condition and $\mathcal{N}_{x_i,0}(1)\geq -Y$. Then for each $\sigma \in (0,1]$,
    $$\{y \in \mathcal{X} \arrowvert r'_{\operatorname{Rm}}(y)<\sigma \}\cap P^{*}(x;A)\cap \mathfrak{t}^{-1}([-A,-A^{-1}])$$ may be covered by at most $C\sigma^{-n+2}$ $P^{\ast}$-parabolic balls of radius $\sigma$.
\end{Corollary}

As an application, we give a new proof of the following estimate for singular K\"ahler Ricci solitons arising as singularity models of Fano K\"ahler Ricci flows. This was originally established in \cite[Theorem 4.31]{CW1}. The proof in \cite{CW1} relied on a rigidity property for projective K\"ahler manifolds \cite{donaldrigid}, whereas our proof uses ideas from \cite{ChNaber2} and only uses the Fano condition to verify that the finite energy condition is satisfied. For this reason, we expect our method may apply in more general situations. We refer the reader to Definition \ref{def:curvaturescale} for the precise definition of $r_{\operatorname{Rm}}$, which is a generalization of the curvature scale to singular spaces.

\begin{Theorem} \label{thm:Kahler} Suppose $(M^n,(g_t)_{t\in [0,T)})$ is a K\"ahler Ricci flow on a K\"ahler manifold $(M,J)$ whose initial K\"ahler class $\omega_0 := g(J\cdot,\cdot)$ satisfies $\omega_0 \in \lambda c_1(M)$ for some $\lambda>0$. There exists $C<\infty$ such that if $X$ denotes any Gromov-Hausdorff limit of $(M,(T-t)^{-\frac{1}{2}}d_{g_t})$ as $t \to T$, then
$$\mathcal{H}^{n}(\{ r_{\operatorname{Rm}}<\sigma \}) \leq C\sigma^4.$$
\end{Theorem}

\begin{Remark} Letting $\mathcal{S}_X$ denote the set of points in $X$ which do not admit neighborhoods locally isometric to a smooth Riemannian $n$-manifold, Theorem \ref{thm:Kahler} implies
$$\mathcal{H}^{n}(\{ x \in X \: : \: d(x,\mathcal{S}_X)<r\}) \leq Cr^4$$
for all $r\in (0,1]$. 
\end{Remark}

\begin{Remark} It is known that given the assumptions of Theorem \ref{thm:Kahler}, $X$ is a singular K\"ahler Ricci soliton uniquely determined by the underlying complex manifold $(M,J)$ \cite{BamScal,CW1,CSW,HanLi}, though this is not needed for the proof of Theorem \ref{thm:Kahler}. Using Theorem \ref{thm:Kahler} and the strategy of \cite{LiuSzek}, one can give an alternative proof of the (known) fact that $X$ is naturally homeomorphic to a projective algebraic variety.
\end{Remark}

The strategy for proving Theorem \ref{thm:epsreg} is similar to that of \cite[Theorem $7.2$]{CCT}, in which one foliates the manifold by level sets of almost splitting and almost radial functions. Let $y=(y^1,...,y^{n-4})$  denote  almost-splitting maps, let $f$ denote an almost-soliton potential, and let $q=4\tau(f-W)-\sum_{i=1}^k y_i^2$  where $W\coloneqq \mathcal{N}_{x_0,t_0}(r^2)$ (c.f. \cite[Proposition $13.1$]{Bam3}). Roughly speaking, the flow nearby is close to the static flow on a cone of the form $C(\mathbb{S}^3/\Gamma)\times \mathbb{R}^{n-4}$, $y$ is close to projection onto the Euclidean factor, and $q$ is close to the squared radial distance from the vertex on the $C(\mathbb{S}^3/\Gamma)$ factor. We will consider slices of the form $$ \Sigma_{z,\lambda ,t}\coloneqq  y_t^{-1}(z)\cap q_t^{-1}(-\infty , \lambda^2 ]\cap P^\ast (x_0,t_0;\Lambda),$$
where $\Lambda \gg \lambda$. Roughly speaking, these sets are close to a ball in some flat cone $C(\mathbb{S}^3/\Gamma)$ of radius $\lambda$. The strategy will be to compute certain Chern-Simons differential characters evaluated on $\partial \Sigma_{z,\lambda,t}$ (which are valued in $\mathbb{R}/\mathbb{Z}$) in two different ways. On one hand, the small $L^2$ curvature assumption on $\Sigma_{z,\lambda,t}$ can be used to show that the differential characters of the boundary $\partial \Sigma_{z,\lambda,t}$ are close to zero. On the other hand, our analysis will show that $\partial \Sigma_{z,\lambda,t}$ (with appropriate connections) converge smoothly to $\mathbb{S}^3/\Gamma$, so that the differential characters converge as well. However, the differential characters of $\mathbb{S}^3/\Gamma$ are already known, so reconciling these facts will yield Theorem \ref{thm:epsreg} and Remark \ref{rem:riemannian}. The main technical difficulty arises from the fact that it is not clear whether $\Sigma_{z,\lambda,t}$ is compactly contained in $P^\ast(x_0,t_0;\Lambda)$. This is due to the fact that $q$ and $y$ are only known to be close to their corresponding models on the smooth part of $C(\mathbb{S}^3/\Gamma) \times \mathbb{R}^{n-4}$. 

The precise statement required to finish the proof of Theorem \ref{thm:epsreg} is that of Lemma \ref{lem: techlemma}. There are three key ingredients, after which the result will follow from standard contradiction-compactness arguments. First, we prove an improved compactness of points for almost-static Ricci flows, strengthening \cite[Theorem $6.49$]{Bam2}. These results are Lemmas \ref{lem:staticpoints} and \ref{lem: staticpoints2}. Secondly, we introduce a notion of uniform convergence of functions with respect to a correspondence of metric flows, and give a criterion for when this convergence holds. This is Lemma \ref{lem: uniform_cont}. This result is not immediately applicable to our setting as $q$ does not solve a forward parabolic equation. The third ingredient is thus to construct parabolic regularizations of $q$, which we refer to as strong $(k,\delta,r)$-soliton potentials, which satisfy the aforementioned convergence criterion. This is Proposition \ref{prop: strong_potentialsl}. We expect these ingredients may be useful, for technical purposes, to future study of non-collapsed limits of Ricci flows. 

Given Theorem \ref{thm:epsreg}, the proof of Theorem \ref{thm:smoothminkowski} follows the strategy of \cite{ChNaber2}, and Theorem \ref{thm:Kahler} then follows from this along with Perelman's estimates for Fano K\"ahler Ricci flow and some elementary considerations. 

The structure of our paper is as follows. In Section \ref{section:preliminaries} we fix notation and recall important definitions and results related to Ricci flow, as well as some topological facts that will be used later. In Section \ref{section: strong_potentials}, we construct strong $(k,\delta,r)$-soliton potentials. In Section \ref{section:bubble} we use the constructions of Section \ref{section: strong_potentials} to prove the main technical results needed for the proof of Theorem \ref{thm:epsreg}. In Section \ref{section:epsreg}, we prove Theorem \ref{thm:epsreg}. In Section \ref{section : Minkowski_estimates}, we prove Theorem \ref{thm:smoothminkowski}, Corollary \ref{cor : limitminkowski}, and Theorem \ref{thm:Kahler}.

%% file: Preliminaries.tex
\section{Notations and Preliminaries}
\label{section:preliminaries}
\subsection{Notation}

Throughout the remainder of the paper, we adhere to the following notational conventions. The notation $A<\infty$ means that $A$ is a large positive constant, while $\epsilon>0$ means that $\epsilon$ is a small positive constant. We let $\Psi(a_1 ,...,a_k|b_1,...,b_{\ell})$ denote a quantity depending on parameters $a_1,...,a_k,b_1,...,b_{\ell}$, which satisfies
$$\lim_{(a_1,...,a_k)\to (0,...,0)} \Psi(a_1,...,a_k|b_1,...,b_{\ell})=0$$
for any fixed $b_1,...,b_{\ell}$. Also, if we say that a proposition $P(\epsilon)$ depending on a parameter $\epsilon$ holds if $\epsilon \leq \overline{\epsilon}(b_1,...,b_{\ell})$, this means there exists a constant $\overline{\epsilon}$ depending on parameters $b_1,...,b_{\ell}$ such that $P(\epsilon)$ holds whenever $\epsilon \in (0,\overline{\epsilon}]$. The notation $E \geq \underline{E}(b_1,...,b_{\ell})$ is defined analogously. The precise value of the constants present in our results will never be of relevance. As a result, when estimating, we will allow constants to change from one line to the next, without introducing new notation. We also omit dependence on the dimension $n$ throughout. 

\subsection{Setup and Basic Concepts}

Our central object of study will be closed Ricci flows $(M^n,(g_t)_{t\in I})$, where $M^n$ is a closed $n$-dimensional manifold, $I\subset \mathbb{R}$ is an interval, and $(g_t)_{t\in I}$ is a smooth family of Riemannian metrics on $M$ which satisfy $$\partial_t g_t=-2Rc_{g_t}.$$ 

We let $d_{g_t}$ be the length metric of $(M,g_t)$, and for $(x,t)\in M \times I$ and $r>0$, we let
$$B(x,t,r):= \{ y \in M \: | \: d_{g_t}(x,y)<r\}$$
denote the corresponding coordinate ball. Define the curvature scale of $(x,t) \in M\times I$ to be 
$$r_{\operatorname{Rm}}(x,t)\coloneqq \sup \left\{ r>0\,\vert \, |Rm|(x',t')\leq r^{-2} \text{ for all } (x',t') \in B(x,t,r) \times ([t-r^2,t+r^2]\cap I)  \right\}.$$ Given  $Y\subseteq M$, we let $|Y|_{g_t}$ denote the volume of $Y$ with respect to the metric $g_t$. Given a space time subset $\widetilde{Y}\subseteq M\times I$, we let $$|\widetilde{Y}| \coloneqq \int_I |\widetilde{Y}\cap (M\times \{t\})|_{g_t}\, dt $$ denote the parabolic volume.

Given a metric space $(X,d)$, we let $\mathcal{H}^k$ denote the corresponding $k$-dimensional Hausdorff measure. For $x\in X$ and $r>0$, we let $B_X(x,r)$ denote the corresponding metric ball. 

\subsection{Conjugate Heat Kernels}

We let
$$\Box \coloneqq  \partial_t-\Delta_{g_t}, \qquad \Box^{\ast}\coloneqq -\partial_t-\Delta_{g_t}+R_{g_t}$$
denote the heat operator and its formal adjoint, respectively. Given $(x_{0},t_{0})\in M \times I$ the \textit{conjugate heat kernel} based at $(x_0,t_0)$ is the function $K(x_{0},t_{0};y,s)\in C^{\infty}(M\times (-\infty,t_{0})\cap I)$ defined by $$ \begin{cases}
        \lim_{s\nearrow t_0}  K(x_0,t_0;y,s)=\delta_{x_0} \\ \Box^{\ast}_{y,s}K(x_0,t_0;y,s)=0
    \end{cases}.$$
The weighted measures $\nu_{x_{0},t_{0};s}\coloneqq K(x_{0},t_{0};\cdot ;s)dg_{s}$ are probability measures, known as the \textit{conjugate heat kernel measures} based at $(x_0,t_0)$. It is often convenient to write $K(x_0,t_0;y,t)=(4\pi\tau)^{-n/2}e^{-f}$, where $\tau\coloneqq t_{0}-t$, and $f\in C^{\infty}(M)$. From this, there is an associated quantity
\begin{align} \label{eq:nashentropy}
    \mathcal{N}_{x_0,t_0}(\tau)\coloneqq \int_{M} f(y,t_0-\tau)\,d\nu_{x_0,t_0;t_0-\tau}(y)-\frac{n}{2},
\end{align}
known as the \textit{pointed Nash entropy} based at $(x_0,t_0)$. It is well known that for each $(x_0,t_0)\in M\times I$, the map $\tau \xrightarrow[]{} \mathcal{N}_{x_0,t_0}(\tau)$ is non-increasing \cite[Proposition $5.2$]{Bam1} . We will often assume a bound of the form $\mathcal{N}_{(x_0,t_0)}(r^{2})\geq -Y$, which can be interpreted as a volume-noncollapsing condition near $(x_0,t_0)$ at scale $r$ (c.f. \cite[Theorem 6.1 and Theorem 8.1]{Bam1}).

We refer the reader to \cite[Sections 2 and 3]{Bam1} for the definition of the Wasserstein distance $d_{W_1}$ and related notions, and their properties on a Ricci flow. We let $d_{W_1}^{g_t}$ denote the Wassserstein distance with respect to the metric $g_t$. Given $(x_0,t_0), (x_1,t_1)\in M\times I$, the map
  \begin{align*}
      t\to d_{W_1}^{g_t}(\nu_{x_0,t_0;t},\nu_{x_1,t_1;t})
  \end{align*}
  is non-decreasing \cite[Lemma $2.7$]{Bam1}. For this reason, the following is a well-behaved replacement for the usual notion of parabolic balls.
  \begin{Definition} \label{def:Pstarparabolic}
      Let $(x_0,t_0)\in M\times I$. For $A,T^+,T^->0$, we define  $$P^{\ast}(x_0,t_0;A,T^-,T^+)=\{ (x,t)\in M\times [t_0-T^-,t_0+T^+]\, \arrowvert \,d_{W_1}^{g_{t_{0}-T^-}}(\nu_{x_0,t_0;t_0-T^-},\nu_{x,t;t_0-T^-})<A\}.$$ 
      We refer to $P^{\ast}(x_0,t_0;r)\coloneqq P^{*}(x_0,t_0;r,r^2,r^2)$ as the $P^*$-parabolic ball of radius $r>0$ around $(x_0,t_0)$, and set $P^\ast_t(x_0,t_0;r)\coloneqq P^\ast(x_0,t_0;r)\cap (M\times \{t\})$.
      
  \end{Definition}
 
  The conjugate heat kernel measure based at $(x_0,t_0)$ may be viewed as a probability distribution for the location of $x_0$ at previous times. The following notion can be thought of as a choice of ``mean" for $\nu_{x_0,t_0;t}$. 
\begin{Definition}
    Let $(x_0,t_0)\in M\times I$, and set $H_n\coloneqq \frac{(n-1)\pi^2}{2}+4$. A point $(z,t)\in M\times I\cap (-\infty,t_0]$ is called an $H_n$-center of $(x_0,t_0)$ if 
  $$\int_M d_{g_t}^2(z,y)d\nu_{x_0,t_0;t}(y) \leq H_n(t_0-t).$$
\end{Definition}
It follows from a concentration bound that $H_n$-centers always exist at each time $t\leq t_0$  \cite[Proposition $3.12$]{Bam1}, and that any two $H_n$-centers $(z,t),(z',t)$ satisfy $d_{g_t}(z,z')\leq 2\sqrt{H_n(t_0-t)}$. 

\subsection{Integral Almost Properties}

Next we recall the integral-almost properties introduced in \cite{Bam3} and \cite{HJ}. These measure how close the flow is to a particular model solution near a given point. We will exclusively work with the strong version of these statements. As in \cite[Proposition $4.1$]{HJ}, one can always construct strong potentials out of weak ones via a regularization procedure. We thus lose no generality by doing this.

Throughout this subsection, we let $(M,(g_t)_{t\in I})$ denote a compact $n$-dimensional Ricci flow. 

\begin{Definition} \label{def:selfsimilar}
     Let $(x_0,t_0)\in M \times I$, $\epsilon,r>0$, and write $W\coloneqq \mathcal{N}_{x_0,t_0}(r^2)$. We say that $(x_0,t_0)$ is $(\epsilon ,r)$-selfsimilar if $[t_0-\epsilon^{-1}r^2,t_0]\subseteq  I $ and 
         \begin{enumerate}
             \item \label{iden:selfsimilar1} $\int_{t_0-\epsilon^{-1}r^{2}}^{t_0-\epsilon r^{2}}\int_M \tau |Rc+\nabla^{2}f-\frac{1}{2\tau }g|^{2}\,d\nu_{x_0,t_0;t}\,dt<\epsilon$,
             \item \label{iden:selfsimilar2} $\sup _{t\in [t_0-\epsilon^{-1}r^{2},t_0-\epsilon r^{2}]}\int_M |\tau(R+2\Delta f-|\nabla f|^{2})+f-n-W|\,d\nu_{x_0,t_0;t}<\epsilon,$
             \item \label{iden:selfsimilar3} $\inf_{M\times [t_0-\epsilon^{-1}r^{2},t_0-\epsilon r^{2}]}R(x,t)\geq -\epsilon r^{2}$
         \end{enumerate}
   where $f$ is given by $d\nu_{x_0,t_0;t}=(4\pi \tau)^{-\frac{n}{2}}e^{-f}dg_t$.
\end{Definition}

\begin{Definition} \label{def : strong_potential}
    Let $(x_0,t_0)\in M \times I$, $\epsilon,r>0$, and write $W\coloneqq \mathcal{N}_{x_0,t_0}(r^2)$. We say that $(x_0,t_0)$ is strongly $(\epsilon,r)$-selfsimilar if it is $(\epsilon,r)$-selfsimilar and there exists $f'\in C^\infty (M\times [t_0-\epsilon^{-1}r^2,t_0-\epsilon r^2])$ satisfying the following:
    \begin{enumerate}
        \item $\Box(4\tau (f'-W))=-2n$ \label{iden : strong_potential11},
        \item $ r^{-2} \int_{t_0-\epsilon^{-1}r^{2}}^{t_0-\epsilon r^{2}}\int_M |\tau(R+|\nabla f'|^2)- (f'-W)| \,d\nu_{x_0,t_0;t}dt < \epsilon \label{iden : strong_potential12}$
        \item $  \int_{t_0-\epsilon^{-1}r^{2}}^{t_0-\epsilon r^{2}}\int_M \tau |Rc+\nabla^2 f'-\frac{1}{2\tau }g|^{2}\,d\nu_{x_0,t_0;t}\,dt<\epsilon, \label{iden : strong_potential13}$
        \item $\sup _{t\in [t_0-\epsilon^{-1}r^{2},t_0-\epsilon r^{2}]}\int_M |\tau(R+2\Delta f'-|\nabla f'|^{2})+f'-n-W|\,d\nu_{x_0,t_0;t}<\epsilon, \label{iden : strong_potential14}$ 
        
        \item $\int_M (f'-\frac{n}{2})\,d\nu_{x_0,t_0;t}=W$ for all $t\in [t_0-\epsilon^{-1}r^2,t_0-\epsilon r^2]$ \label{iden : strong_potential15}.
    \end{enumerate}
    
    A function $f'$ that satisfies \ref{iden : strong_potential11}-\ref{iden : strong_potential15} is referred to as a strong $(\epsilon,r)$-soliton potential.
\end{Definition}

The following is a consequence of \cite[Proposition 3.4]{HJ}, the Poincare inequality \cite[Theorem 11.1]{Bam1}, and Definition \ref{def : strong_potential}\ref{iden : strong_potential15}.
\begin{Proposition} \label{prop:boundforf} For any $\epsilon>0$, $p\in [1,\infty)$, and $Y<\infty$, the following holds if $\delta \leq \overline{\delta}(\epsilon,Y,p)$. Suppose $(x_0,t_0)\in M\times I$, $r>0$, satisfy $[t_0-\delta^{-1}r^2,t_0]\subseteq I$ and $\mathcal{N}_{x_0,t_0}(r^2)\geq -Y$. If $f'\in C^{\infty}(M\times [t_0-\delta^{-1}r^2,t_0-\delta r^2])$ is a strong $(\delta,r)$-soliton potential, then 
$$\sup_{t\in [t_0-\epsilon^{-1}r^2,t_0-\epsilon r^2]} \int_M (1 +|f'|+|\nabla f'|)^p d\nu_t \leq C(Y,\epsilon,p).$$
\end{Proposition}
    
The following observation will be needed in the proof of Proposition \ref{prop: strong_potentialsl}.

\begin{Proposition} \label{prop:equiv}
	For any $Y<\infty$ and $\epsilon>0$, the following holds if $\delta \leq \overline{\delta}(Y,\epsilon)$. Suppose $(x_0,t_0)\in M\times I$, $r>0$ satisfy $\mathcal{N}_{x_0,t_0}(r^2)\geq -Y$. If $(x_0,t_0)$ is strongly $(\delta,r)$-selfsimilar, then any strong $(\delta,r)$-soliton potential $f^\prime$ satisfies 
    $$\sup_{[t_0-\epsilon^{-1}r^2,t_0-\epsilon r^2]}\int_M (f-f^\prime)^2d\nu_t+ \int_{t_0-\epsilon^{-1}r^2}^{t_0-\epsilon r^2}\int_M |\nabla (f-f^\prime)|^2 d\nu_tdt \leq \epsilon $$ 
    where $f$ is given by $d\nu_{x_0,t_0;t}=(4\pi \tau)^{-\frac{n}{2}}e^{-f}dg_t$.
	\end{Proposition}
	\begin{proof}
	By parabolic rescaling and a a time-shift, we may assume that $r=1$ and $t_0=0$. We have 
		\begin{align}
			\int_{-2\epsilon^{-1} }^{-\epsilon} \int_M |\nabla (f-f^\prime)|^2 d\nu_tdt &= 	\int_{-2\epsilon^{-1} }^{-\epsilon} \int_M \left(  |\nabla f|^2 + |\nabla f^\prime|^2-2\langle \nabla f ,\nabla f^\prime \rangle  \right)d\nu_tdt \nonumber \\ & =\int_{-2\epsilon^{-1} }^{-\epsilon} \int_M \left(  -\Delta_{f^\prime }f^\prime+(\Delta f-\Delta f^\prime) \right)d\nu_tdt \label{eq:equiv1}
					\end{align}
			Moreover by Definition \ref{def : strong_potential}\ref{iden : strong_potential13} and \cite[Proposition $7.3$]{Bam3}, and the fact that $(x_0,0)$ is $(\delta,1)$-selfsimilar by definition, we have
			\begin{align}
				\int_{-2\epsilon^{-1} }^{-\epsilon} \int_M |\Delta f-\Delta f^\prime| d\nu_tdt\leq \int_{-2\epsilon^{-1} }^{-\epsilon} \int_M \left[ |R+\Delta f-\frac{n}{2\tau }|+|R+\Delta f^\prime-\frac{n}{2\tau }|\right] d\nu_tdt\leq \Psi(\delta \arrowvert \epsilon ,Y) \label{eq:equiv2}.
			\end{align}
			Since 
			$$ 2\tau \Delta_{f'} f'= \left( \tau (R+\Delta f') - \frac{n}{2} \right) -\left( \tau(R+|\nabla f'|^2) -(f' -W)\right)+(f'-\frac{n}{2}-W),$$
			we get from Definition \ref{def : strong_potential}\ref{iden : strong_potential12},\ref{iden : strong_potential13},\ref{iden : strong_potential15} and Proposition \ref{prop:boundforf} that
			\begin{align}
				\int_{-2\epsilon^{-1} }^{-\epsilon} \int_M \Delta_{f'} f'\,d\nu_tdt&=\frac{1}{2\tau}	\int_{-2\epsilon^{-1} }^{-\epsilon} \int_M \left(\tau (R+\Delta f'-\frac{n}{2\tau})\right)-\left( \tau(R+|\nabla f'|^2) -f' +W\right)d\nu_tdt\nonumber \\ &\leq \Psi(\delta\arrowvert \, Y,\epsilon) \label{eq:equiv3}.
			\end{align}
			Combining (\ref{eq:equiv1})-(\ref{eq:equiv3}) gives that 
			\begin{align}
			\int_{-2\epsilon^{-1}}^{-\epsilon }\int_M |\nabla (f-f^\prime)|^2 d\nu_tdt \leq \Psi(\delta \arrowvert Y,\epsilon).	\label{eq:equiv5}
			\end{align}
            Applying the weighted Poincare inequality [Hein-Naber], \cite[Proposition 7.1]{Bam3}, and Definition \ref{def : strong_potential}\ref{iden : strong_potential15} yields
			\begin{align}
					\int_{-2\epsilon^{-1}}^{-\epsilon }\int_M  (f-f^\prime)^2 d\nu_tdt  \leq \Psi(\delta \arrowvert Y,\epsilon)+2\int_{-2\epsilon^{-1}}^{-\epsilon }\left(\int_M (f-f^\prime )d\nu_t\right)^2dt\leq  \Psi(\delta \arrowvert \,Y,\epsilon) \label{eq:equiv6}
			\end{align}
			Moreover, by \cite[Proposition $6.2$]{Bam3} we have
			\begin{align}
					\int_{-2\epsilon^{-1}}^{-\epsilon }\int_M |\Box(f-f^\prime)|^2d\nu_t dt&= 	\int_{-2\epsilon^{-1}}^{-\epsilon }\int_M \left( -\Delta f+|\nabla f|^2-\frac{n}{2\tau} +\frac{-n+2(f^\prime-W)}{2\tau}\right)^2\,d\nu_t dt \nonumber  \\ & \leq C(Y). \label{eq:equiv7}
				\end{align}
			Let $\xi: [-2\epsilon ^{-1},-\epsilon]$ be a smooth cutoff function such that $\xi(-2\epsilon^{-1})=0$, $\xi\arrowvert_{[-\epsilon^{-1},-\epsilon]}\equiv 1$, and $|\xi^\prime \tau|\leq C$ for some universal constant $C<\infty $. By combining (\ref{eq:equiv5})-(\ref{eq:equiv7}), we get that 
			\begin{align*}
				\int_M (f-f^\prime)^2\,d\nu_{t_1}&=\int_{-2\epsilon^{-1}}^{t_1}\xi^\prime \int_M (f-f^\prime)^2d\nu_tdt+\int_{-2\epsilon^{-1}}^{t_1}\xi \int_M (f-f^\prime)\Box(f-f^\prime)d\nu_tdt \\&+\int_{-2\epsilon^{-1}}^{t_1}\xi \int_M |\nabla (f-f^\prime)|^2 d\nu_tdt\\ & \leq \Psi(\delta \arrowvert \epsilon ,Y)
				\end{align*}
				for every $t_1\in [-\epsilon^{-1},-\epsilon]$. This completes the proof.
	\end{proof}

\begin{Definition} \label{def : strong_splitting}
    Let $(x_0,t_0)\in M \times I$ and $\epsilon,r>0$. We say that $(x_0,t_0)$ is strongly $(k,\epsilon,r)$-split if $[t_0 - \epsilon^{-1}r^2,t_0] \subseteq I$ and there exist $y_1,...,y_k \in C^{\infty}(M \times [t_0 - \epsilon^{-1}r^2,t_0-\epsilon r^2])$ such that the following hold for $1\leq i,j\leq k$:
    \begin{enumerate}
        \item $\Box y_i=0, \label{iden : strong_splitting1}$ 
        \item $r^{-2} \int_{t_0-\epsilon^{-1}r^2}^{t_0-\epsilon r^2} \int_M |\langle \nabla y_i, \nabla y_j\rangle - \delta_{ij}| d\nu_{x_0,t_0;t}dt\leq \epsilon$ \label{iden : strong_splitting2},
        \item $\int_M y_i \, d\nu_{x_0,t_0;t}=0 \text{ for each }t\in [t_0-\epsilon^{-1} r^2,t_0-\epsilon r^2]$ \label{iden : strong_splitting3}.
    \end{enumerate}
   A tuple $(y^1,...,y^k)$ that satisfies \ref{iden : strong_splitting1}-\ref{iden : strong_splitting3} is referred to as a strong $(k,\epsilon,r)$-splitting map. 
\end{Definition} 

We now recall some properties of strong splitting maps and strong almost-soliton potentials which will be used frequently in later sections. 

\begin{Proposition} \label{prop:splittingproperties} For any $Y<\infty$ and $\epsilon>0$, the following hold if $\delta \leq \overline{\delta}(\epsilon,Y)$. Suppose $(x_0,t_0)\in M\times I$, $r>0$, satisfy $[t_0-\delta^{-1}r^2,t_0]\subseteq I$ and $\mathcal{N}_{x_0,t_0}(r^2)\geq -Y$. Write $d\nu_{x_0,t_0;t}=(4\pi \tau)^{-\frac{n}{2}}e^{-f}dg_t$, and suppose $y=(y_1,...,y_k)\in C^{\infty}(M\times [t_0-\delta^{-1}r^2,t_0-\delta r^2],\mathbb{R}^k)$ is a strong $(k,\delta,r)$-splitting map. Then the following hold:
\begin{enumerate}
    \item \label{eq:splittingproperty1} For any $p\in [1,10]$ and $t\in [t_0-\epsilon^{-1}r^2,t_0-\epsilon r^2]$, 
    $$\sum_{i,j=1}^k \int_M |\langle \nabla y_i, \nabla y_j \rangle-\delta_{ij}|^p d\nu_{x_0,t_0;t} <\epsilon,$$
    \item \label{eq:splittingproperty2} For any $t\in [t_0-\epsilon^{-1}r^2,t_0-\epsilon r^2]$, 
    $$\left| \int_M y_i^2 d\nu_{x_0,t_0;t} - 2\tau \right| < \epsilon,$$
    \item \label{eq:splittingproperty3} For any $p,q\in [0,20]$, we have
    $$r^{2-q} \int_{t_0-\epsilon^{-1}r^2}^{t_0-\epsilon r^2} \int_M \left( \sum_{i=1}^k |\nabla^2 y_i|^2 \right)\left( \sum_{j=1}^k |\nabla y_j|^p \right) \left( \sum_{\ell=1}^k |y_{\ell}|^q \right) d\nu_t dt <\epsilon.$$
    \item \label{eq:splittingproperty4} If in addition $(x_0,t_0)$ is strongly $(\delta,r)$-selfsimilar, and $f'\in C^{\infty}(M\times [t_0-\delta^{-1}r^2,t_0-\delta r^2])$ is a strong $(\delta,r)$-soliton potential, then
    for any $p \in [1,4]$ and $q\in [0,4]$, we have
    $$r^{-p-q}\int_{t_0-\epsilon^{-1}r^2}^{t_0-\epsilon r^2} \int_M |\langle \nabla (4\tau f'),\nabla y_i\rangle -2y_i|^{p} \left( \sum_{j=1}^k |y_j|^q\right) d\nu_t dt<\epsilon.$$
\end{enumerate}
\end{Proposition}
\begin{proof} Assertions \ref{eq:splittingproperty1}-\ref{eq:splittingproperty3} follow directly from \cite[Proposition 12.21]{Bam3}. In the special case where $q=0$ and $p\in [1,2]$, \ref{eq:splittingproperty4} is a consequence of \cite[Proof of Proposition $5.3(ii)$]{HJ} and Proposition \ref{prop:equiv}. This implies the general case when combined with Cauchy's inequality, \ref{eq:splittingproperty1},\ref{eq:splittingproperty3}, and Proposition \ref{prop:boundforf}.
\end{proof}

\begin{Definition} \label{def : static }
    Let $(x_0,t_0)\in M\times I$ and  $\epsilon,r>0$. We say that $(x_0,t_0)$ is $(\epsilon,r)$-static, if $[t_0-\epsilon^{-1}r^2,t_0]\subseteq I$ and the following holds 
    \begin{enumerate}
        \item $r^{2} \int_{t_0-\epsilon^{-1}r^{2}}^{t_0-\epsilon r^{2}}\int_M|Rc|^2\,d\nu_{x_0,t_0;t}\,dt<\epsilon$ \label{iden : strong_static1},
        \item $\sup_{[t_0-\epsilon^{-1} r^2, t_0-\epsilon r^2]}\int _{M} R\, d\nu_{x_0,t_0;t}<\epsilon$ \label{iden : strong_static2},
        \item $\inf_{M\times [t_0-\epsilon^{-1} r^2, t_0-\epsilon r^2] }R(x,t)\geq -\epsilon r^{-2}$ \label{iden : strong_static3}.
    \end{enumerate}
\end{Definition}

\subsection{Metric Flows and the $\mathbb{F}$-Topology}
\label{subsection:metricflows}

Throughout this article, we will make frequent use of contradiction-compactness arguments. This will usually involve limits of Ricci flows which are not smooth, and are instead objects called \textit{metric flows}. We mostly refer the reader to \cite[Section 3]{Bam2} for the complete definitions and basic properties, and just recall notation here. A metric flow over an interval $I \subseteq \mathbb{R}$ will consist of data (satisfying additional properties listed in \cite[Definition 3.2]{Bam2}) $(\mathcal{X},(d_t)_{t\in I},(\nu_{x;s})_{x\in \mathcal{X},s\in I\cap (-\infty, \mathfrak{t}(x))},\mathfrak{t})$, where $\mathcal{X}$ is the spacetime set, $\mathfrak{t}:\mathcal{X} \to I$ is a time function, $d_t$ are metrics on the time slices $\mathcal{X}_t := \mathfrak{t}^{-1}(t)$, and for each $x\in \mathcal{X}_t$ and $s\in (-\infty,t)\cap I$, $\nu_{x;s}$ is a Borel probability measure on $(\mathcal{X}_t,d_t)$. 

In \cite{Bam2,Bam3}, a notion of closeness of Ricci flows (or more generally metric flows) was introduced, and the relevant compactness and partial regularity theories were developed. Given metric flows $\mathcal{X}_i$ with reference conjugate heat flows $(\mu_t^i)_{t\in I}$, $i=1,2$, there is a notion of distance $d_{\mathbb{F}}((\mathcal{X}_1,(\mu_t^1)_{t\in I}),(\mathcal{X}_2,(\mu_t^2)_{t\in I})$ between the flows (see \cite[Definition 5.6]{Bam2}), which can be seen as a Ricci flow analogue of pointed Gromov-Hausdorff distance (where $\mu_t^i$ play the role of basepoints). Moreover, given a sequence of metric flow pairs $(\mathcal{X}^i,(\mu_t^i)_{t\in I_i})$ converging with respect to $d_{\mathbb{F}}$, we can fix a \textit{correspondence}  
\begin{equation} \label{eq:correspondence} \mathfrak{C} := \left( (Z_t,d_{Z_t})_{t\in I}, (\varphi_t^i)_{t\in I_i,i\in \mathbb{N}}, (\varphi_t)_{t\in I} \right)\end{equation}
which realizes this convergence (see \cite[Definition 6.2]{Bam2}). Here, $(Z_t,d_t)$ are complete metric spaces and $\varphi_t^i:\mathcal{X}_t^i \to Z_t$, $\varphi_t:\mathcal{X}_t \to Z_t$ are isometric embeddings. This data is analogous to the choice of isometric embeddings which allow pointed Gromov-Hausdorff convergence to be realized as Hausdorff convergence in a fixed metric space, except that here we use Wasserstein distance, and the convergence must be compatible with other time slices in some sense. In particular, correspondences allow us to makes sense of the notion of convergence of points (and more generally conjugate heat flows) within a correspondence (see \cite[Definition 6.18]{Bam2}). Given a sequence of metric flow pairs $(\mathcal{X}^i,(\mu_t^i)_{t\in I_i})$ which $\mathbb{F}$-converge to a metric flow pair $(\mathcal{X},(\mu_t)_{t\in I})$, we will fix a correspondence $\mathfrak{C}$ realizing this convergence (which exists by \cite[Theorem 6.9]{Bam2}), and write this as
$$(\mathcal{X}^i,(\mu_t^i)_{t\in I_i}) \xrightarrow[i\to \infty]{\mathbb{F},\mathfrak{C}} (\mathcal{X},(\mu_t)_{t\in I}).$$
If a sequence of points $x_i \in \mathcal{X}_i$ converges to a point $x\in \mathcal{X}$ in this same correspondence, we write
$$x_i \xrightarrow[i\to \infty]{\mathfrak{C}} x.$$

For our main results, we will only need a special case of this compactness theory, where the limiting metric flow $\mathcal{X}$ is a \textit{metric soliton} (c.f. \cite[Definition 3.57]{Bam2}). In this case, there is an identification $\mathcal{X} \cong X \times (-\infty,0)$, where $X$ is a singular shrinking gradient Ricci soliton with singularities of Minkowski codimension four. In particular, $\mathcal{X}$ is entirely determined by $X$. 

\begin{Theorem}[c.f. Theorems 7.6, 9.12, 9.31 in \cite{Bam2}] \label{bamconvergence} Suppose $(M_i^n,(g_{i,t})_{t\in I_i})$ is a sequence of pointed Ricci flows and $x_i \in M_i$, $\delta_i \searrow 0$ are such that $(x_i,0)$ are strongly $(\delta_i,1)$-selfsimilar and $\mathcal{N}_{x_i,0}(1)\geq -Y$ for some $Y<\infty$, with strong $(\delta_i,1)$-soliton potentials $f_i'$. Then we can pass to a subsequence to obtain a metric soliton $(\mathcal{X},(\mu_t)_{t\in (-\infty,0]})$ along with a correspondence $\mathfrak{C}$ such that we have the following uniform $\mathbb{F}$-convergence on compact time intervals:
\begin{align}
    (M_i^n,(g_{i,t})_{t\in (-T_i,0]},(\nu_{x_i,0;t})_{t\in I_i})\xrightarrow[i\to \infty]{\mathbb{F},\mathfrak{C}} (\mathcal{X},(\mu_t)_{t\in (-T,0)}) \label{eq:Fconvergence}.
\end{align}
Moreover, there is a homeomorphism $\mathcal{X} \cong X\times (-\infty,0]$ (where $\mathcal{X}$ is equipped with the natural topology as in \cite[Section 3.6]{Bam2}) for some singular space $(X,d)$ (in the sense of \cite[Definition 2.15]{Bam3}) with regular set $(\mathcal{R}_X,g_X)$, and that the following hold under the identification $\mathcal{X} \cong X\times (-\infty,0)$:
\begin{enumerate}
    \item the regular set $\mathcal{R}$ of the metric flow (see \cite[Definition 9.11]{Bam2} is identified with $\mathcal{R}_X \times (-\infty,0)$ as Ricci flow spacetimes (see \cite[Definition 9.1]{Bam2}), where there exists $f\in C^{\infty}(\mathcal{R})$ such that (setting $\tau:=|t|$)
    $$Rc_{g_t}+\nabla^2 f = \frac{1}{2\tau}g, \qquad d\mu_t = (4\pi \tau)^{-\frac{n}{2}}e^{-f}dg_t$$
    on $\mathcal{R}$, and if $\partial_{\mathfrak{t}}$ is the timelike vector field of $\mathcal{R}$, then $\partial_{\mathfrak{t}}-\nabla f$ is identified with the standard vector field $\partial_t$ on the time factor for some $f\in C^{\infty}(\mathcal{R})$,
    \item there are time-preserving open embeddings $\psi_i:\mathcal{R}\supseteq U_i \to M_i$, where $(U_i)_{i\in \mathbb{N}}$ is a precompact open exhaustion of $\mathcal{R}$, such that
    $$\psi_i^{\ast}g_i \to g, \qquad \psi_i^{\ast}f_i' \to f, \qquad (\psi_i)_{\ast}\partial_t \to \partial_{\mathfrak{t}}$$
    in $C_{\operatorname{loc}}^{\infty}(\mathcal{R})$. 
\end{enumerate}
\end{Theorem}

In the setting of Theorem \ref{bamconvergence}, we can define the following curvature scale function on the shrinking soliton $X$.

\begin{Definition} \label{def:curvaturescale} For $x\in \mathcal{R}_X$, let $r_{\operatorname{Rm}}(x)$ be the infimum of all $r>0$ such that $B_X(x,r) \subseteq \mathcal{R}_X$ and $\sup_{B_X(x,r)}|Rm| \leq r^{-2}$. 
\end{Definition}

\subsection{Chern-Simons Invariants and Chern-Gauss-Bonnet}
\label{section:ChernSimons}

We now summarize some properties of secondary characteristic classes constructed in \cite{chernsimons,cheegersimons}, which will be needed in the proof of Theorem \ref{thm:epsreg}. We focus on two special cases of this construction, which correspond to the second Chern class $c_2$ and the first Pontryagin class $p_1$. Given a smooth manifold $M$, let $C_k(M),Z_k(M)$ denote the smooth singular $k$-chains and those chains which are closed, respectively.  

For any complex vector bundle $\pi:E\to M$ and connection $\nabla$ on $E$, a corresponding differential character $\widehat{c}_2(E,\nabla):Z_3(M) \to \mathbb{R}/\mathbb{Z}$ was constructed in \cite[Section 4]{cheegersimons}. 
The construction of $\widehat{c}_2(E,\nabla)$ is somewhat involved, so we only summarize the properties which we need for the proof of Theorem \ref{thm:epsreg}.

\begin{Proposition} \label{prop:secondchernproperties}
\begin{enumerate}
    \item \label{secondchern1} (Naturality) If $E\to M$ and $E'\to M'$ are complex vector bundles, $\nabla'$ is a connection on $E'$, and $\Phi:E\to E'$ is a bundle isomorphism over $\phi:M\to M'$, then 
    $$\langle \widehat{c}_2(E,\Phi^{\ast}\nabla'),\sigma\rangle = \langle \widehat{c}_2(E',\nabla'),\phi_{\ast}\sigma\rangle $$
    for all $\sigma \in Z_3(M)$. 

    \item \label{secondchern2} (Behavior on boundaries) For any $\sigma \in C_4(M)$, we have
    \begin{equation} \label{eq:secondchernboundary} \langle \widehat{c}_2(E,\nabla ),\partial \sigma  \rangle \equiv \frac{1}{8\pi^2} \int_{\sigma} \left( \operatorname{tr}(F_{\nabla})^2-\operatorname{tr}(F_{\nabla}^2) \right) \mod \mathbb{Z}\end{equation}
where $F_{\nabla}$ denotes the curvature 2-form of $\nabla$ with respect to a local frame of the complex tangent bundle of $M$, viewed as a function valued in $ \mathfrak{gl}(2,\mathbb{C})$. 

    \item \label{secondchern3} (Change of connection) For any complex vector bundle $E\to M$ and any connections $\nabla,\nabla'$ on $E$, if we set $A:=\nabla'-\nabla\in \mathcal{A}^1(M,\mathfrak{u}(2))$, then 
\begin{equation}
\label{eq:secondchern}
\begin{aligned}
\langle \widehat{c}_2(E,\nabla')-\widehat{c}_2(E,\nabla),\sigma\rangle
= \frac{1}{4\pi^2}\int_{\sigma} \Big(
& \mathrm{Tr}(A)\wedge \mathrm{Tr}(F_{\nabla})
+ \tfrac{1}{2}\,\mathrm{Tr}(A)\wedge d\,\mathrm{Tr}(A)
+ \tfrac{1}{3}\,\mathrm{Tr}(A)\wedge \mathrm{Tr}(A \wedge A) \\
& - \mathrm{Tr}(A \wedge F_{\nabla})
- \tfrac{1}{2}\,\mathrm{Tr}(A \wedge d_{\nabla}A)
- \tfrac{1}{3}\,\mathrm{Tr}(A \wedge A \wedge A)
\Big)
\end{aligned}
\end{equation}
for $\sigma \in Z_3(M)$, where $d_{\nabla}$ is the connection exterior derivative with respect to $\nabla$. 

\item \label{secondchern4} If $\overline{\nabla}$ denotes the (flat) Chern connection on $E:=T(\mathbb{C}^{n-2} \times (\mathbb{C}^2\setminus \{0\})/\Gamma )|_{\{0\} \times \mathbb{S}^3/\Gamma}$ for some finite subgroup $\Gamma \leq U(2)$ acting freely on $\mathbb{S}^3$, then
$$\langle \widehat{c}_2(E,\overline{\nabla}), \{0\} \times \mathbb{S}^3/\Gamma\rangle \equiv \frac{1}{|\Gamma|} \mod \mathbb{Z}.$$
\end{enumerate}
\end{Proposition}

\begin{Remark}
The precise formula \eqref{eq:secondchern} will not be important for us, but we will use the fact that the integrand is small when $A$ is small. 
\end{Remark}

We now consider more generally the case where $E\to M$ is a real vector bundle with a metric connection $\nabla$. In this case, there is a differential character $\widehat{p}_1(E,\nabla):Z_3(M)\to \mathbb{R}/\mathbb{Z}$, whose properties are analogous to those of $\widehat{c}_2$. 

\begin{Proposition} \label{prop:firstpontryaginproperties}
\begin{enumerate}
    \item \label{firstpontryagin1} (Naturality) If $E\to M$ and $E'\to M'$ are vector bundles, $\nabla'$ is a connection on $E'$, and $\Phi:E\to E'$ is a bundle isomorphism over $\phi:M\to M'$, then 
    $$\langle \widehat{p}_1(E,\Phi^{\ast}\nabla'),\sigma\rangle = \langle \widehat{p}_1(E',\nabla'),\phi_{\ast}\sigma\rangle $$
    for all $\sigma \in Z_3(M)$. 

    \item \label{firstpontryagin2} (Behavior on boundaries) For any $\sigma \in C_4(M)$, we have
    \begin{equation} \label{eq:firstpontryaginboundary} \langle \widehat{p}_1(E,\nabla ),\partial \sigma  \rangle \equiv -\frac{1}{8\pi^2} \int_{\sigma} \operatorname{tr}(F_{\nabla}^2) \mod \mathbb{Z},\end{equation}

    \item \label{firstpontryagin3} (Change of connection) For any vector bundle $E\to M$ and connections $\nabla,\nabla'$ on $E$, if we set $A:=\nabla'-\nabla\in \mathcal{A}^1(M,\mathfrak{u}(2))$, then 
\begin{equation}
\label{eq:firstpontryagin}
\begin{aligned}
\langle \widehat{p}_1(E,\nabla')-\widehat{p}_1(E,\nabla),\sigma\rangle
= -\frac{1}{4\pi^2}\int_{\sigma} \Big( -\left(
\mathrm{Tr}(A \wedge F_{\nabla})
+ \tfrac{1}{2}\,\mathrm{Tr}(A \wedge d_{\nabla}A)
+ \tfrac{1}{3}\,\mathrm{Tr}(A \wedge A \wedge A)
\right).
\Big)
\end{aligned}
\end{equation}
for $\sigma \in Z_3(M)$.

\item \label{firstpontryagin4} If $\overline{\nabla}$ denotes the (flat) Levi-Civita connection on $E:=T(\mathbb{R}^{n-4} \times (\mathbb{R}^4\setminus \{0\})/\Gamma )|_{\{0\} \times \mathbb{S}^3/\Gamma}$ for some finite subgroup $\Gamma \leq O(3,\mathbb{R})$ acting freely on $\mathbb{S}^3$, then
$$\langle \widehat{p}_1(E,\overline{\nabla}), \{0\} \times \mathbb{S}^3/\Gamma\rangle \equiv  0\mod \mathbb{Z}$$
if and only if $\mathbb{S}^3/\Gamma$ is an exceptional lens space $L_{p,q}=\mathbb{S}^3/\mathbb{Z}_p$, where $\mathbb{Z}_p \subseteq \mathbb{C}^{\ast}$ acts on $\mathbb{S}^3 \subseteq \mathbb{C}^2$ with weights $(1,q)$, and $q^2 \equiv -1 \mod p$.
\end{enumerate}
\end{Proposition}

In the four-dimensional case, we will instead use the Chern-Gauss-Bonnet theorem with boundary. This states that for any Riemannian 4-manifold $(M,g)$ and any domain $\Omega \subseteq M$ with smooth boundary, we have
\begin{equation} \label{eq:cherngaussbonnet} 
\begin{aligned}
32\pi^2 \chi(\partial \Omega) =& \int_{\Omega} (|Rm|^2 -4|Rc|^2+R^2)dg + 16\int_{\partial \Omega} k_1 k_2 k_3 d\mathcal{H}_g^3 \\&+8 \int_{\partial \Omega} (k_1 K_{23}+k_2 K_{13} + k_3 K_{12})dg
\end{aligned}
\end{equation}
where $K_{ij}$ denote the ambient sectional curvatures of $(M,g)$ evaluated on planes tangent to $\partial \Omega$, and $k_i$ are the principal curvatures of the embedding $\partial \Omega \hookrightarrow M$. 

%% file: Strong_Potentials.tex
\section{Strong $(k,\delta,r)$-Soliton Potentials} \label{section: strong_potentials}

Consider a (possibly singular) shrinking gradient Ricci soliton of the form $X = X' \times \mathbb{R}^{k}$. Let $y=(y_1,...,y_{k}):X \to \mathbb{R}^{k}$ be the projection map, and define $h:= f-\frac{1}{4\tau}\sum_{j=1}^k y_j^2$. Then for each $z \in \mathbb{R}^k$, $y^{-1}(z)$ with the restricted metric is a (possibly singular) shrinking gradient Ricci soliton with potential function $h:= f- {4\tau}\sum_{j=1}^k y_j^2$. We will now consider a smooth, compact Ricci flow which is close in the $\mathbb{F}$-topology to $X = X' \times \mathbb{R}^k$. Our main result in this section will be the existence of a solution of a forward parabolic equation which almost (in the weighted integral sense) satisfies the same identities as $h$. 

\begin{Definition} \label{def : strong_potential2}
    Let $(M^{n},(g_t)_{t\in I})$ be a closed Ricci flow, $(x_0,t_0)\in M \times I$, $\epsilon,r>0$, and write $W\coloneqq \mathcal{N}_{x_0,t_0}(r^2)$. We say that $(x_0,t_0)$ is strongly $(k,\epsilon,r)$-selfsimilar if it is strongly $(\epsilon,r)$-selfsimilar, strongly $(k,\epsilon,r)$-split, and there exists $h\in C^{\infty}(M\times [t_{0}-\epsilon^{-1}r^{2},t_{0}-\epsilon r^{2}])$ together with strong $(k,\epsilon,r)$-splitting maps $(y^1,...,y^k)$ such that the following hold:
    \begin{enumerate}
        \item \label{eq:strongheateq1} $ \Box (4\tau(h-W)) =-2(n-k),$
        \item \label{eq:strongheateq2} $\int_{t_{0}-\epsilon^{-1}r^{2}}^{t_{0}-\epsilon r^{2}}\int_{M} |\langle 4\tau \nabla h ,\nabla y_{j}\rangle |^{2}\,d\nu_{x_0,t_0;t}dt < \epsilon$ for $1\leq j\leq k$,
        \item \label{eq:strongheateq3}$ \int_{t_{0}-\epsilon^{-1}r^{2}}^{t_{0}-\epsilon r^{2}} \tau\int_{M}|Rc+\nabla ^{2}h-\frac{1}{2\tau }(g-\sum_{j=1}^kdy_{j}\otimes dy_{j})|^{2}\,d\nu_{x_0,t_0;t}dt <\epsilon,$
        \item \label{eq:strongheateq4} $\sup_{[t_{0}-\epsilon^{-1}r^{2},t_{0}-\epsilon r^{2}]}\int_{M}\left| \tau(R+2\Delta h-|\nabla h|^{2})+h-(n-k)-W\right|\,d\nu_{x_0,t_0;t}< \epsilon,$
        \item \label{eq:strongheateq5} $r^{-2}\int_{t_0-\epsilon^{-1}r^2}^{t_0-\epsilon r^2}\int_{M}\left| \tau(R+|\nabla h|^2)-h+W\right|\,d\nu_{x_0,t_0;t}dt< \epsilon,$
        \item \label{eq:strongheateq6} $\int_M (h-\frac{n-k}{2})\,d\nu_{x_0,t_0;t}=W $ for every $t\in [t_0-\epsilon^{-1}r^2,t_0-\epsilon r^2]$.
    \end{enumerate}
    A function $h$ which satisfies \ref{eq:strongheateq1}-\ref{eq:strongheateq6} is referred to as a strong $(k,\epsilon,r)$-soliton potential.
   \end{Definition}

The following proposition is analogous to \cite[Proposition $12.1$]{Bam3} and \cite[Proposition $4.1$]{HJ}, and gives a sufficient criterion for the existence of strong $(k,\epsilon,r)$-soliton potentials.

\begin{Proposition} \label{prop: strong_potentialsl}
    For any $\epsilon>0$ and $Y<\infty$, the following holds if $\delta \leq\overline{\delta}(\epsilon, Y)$. Let $(M^n,(g_t)_{t\in I})$ be a closed Ricci flow, and suppose $(x_0,t_0)\in M\times I$ and $r>0$ satisfy $W:= \mathcal{N}_{x_0,t_0}(r^2)\geq -Y$. If $(x_0,t_0)$ is strongly $(\delta,r)$-selfsimilar and strongly $(k,\delta,r)$-split, with $f'$ and $y=(y^1,...,y^k)$ a strong $(\delta,r)$-soliton potential and strong $(k,\delta,r)$-splitting map, respectively, then $(x_0,t_0)$ is strongly $(k,\epsilon,r)$-selfsimilar, and there is a strong $(k,\epsilon,r)$-soliton potential $h^\prime$ satisfying
    \begin{align} \label{eq : strong_potential1}
        \int_{t_{0}-\epsilon^{-1}r^{2}}^{t_{0}-\epsilon r^{2}} \int_{M}|\nabla (h-h^{\prime})|^{2}d\nu_{x_0,t_0;t}dt+   \sup_{t\in [t_0 -\epsilon^{-1}r^2,t_0 +\epsilon^{-1}r^2]} \int_{M}| h^{\prime}-h|^{2}d\nu_{x_0,t_0;t}dt \leq \epsilon
    \end{align}
    where $h:=f'-\frac{1}{4\tau}\sum_{j}y_{j}^{2}$.
\end{Proposition}

\begin{Remark}
    Similarly to Proposition \ref{prop:equiv}, one can show that (\ref{eq : strong_potential1}) is automatically satisfied by any strong $(k,\delta,r)$-soliton potential, although we will not need this fact.
\end{Remark}

\begin{proof}
By means of parabolic rescaling and time translation, we can assume that $r=1$ and $ t_{0}=0$. For ease of notation, we write $\nu_t:=\nu_{x_0,0;t}$. For $t^\ast \in [-101\epsilon^{-1},-100\epsilon^{-1}]$ to be determined, let  $\tilde{y}_j^2$ be the solution to 
    \begin{align*}
    \begin{cases}
         \Box \tilde{y}_j^2=-2 \\ \tilde{y}_{j,t^{\ast}}^2 = y_{j,t^{\ast}}^2.
    \end{cases}  
    \end{align*}
\begin{Claim} \label{Claim:: strong_potential2}
     We have
$$ \int_{ t^\ast}^{-\frac{1}{4}\epsilon }\int_{M}|\nabla(\tilde{y}_j^2-y_j^2)|^{2}  d\nu_{t}dt+ \sup_{t\in [ t^\ast,-\frac{1}{4}\epsilon ]} \int_{M} (\tilde{y}_j^2-y_j^2)^{2}d\nu_{t}  \leq \Psi(\delta| Y,\epsilon).$$ 
\end{Claim}
\begin{proof}
    
    Take $\sigma>0$ to be determined. For any $t\in [t^{\ast},\epsilon]$, we have the following:
    \begin{align*}
    \frac{d}{dt}\int_{M}  (\tilde{y}_j^2-y_j^2)^{2} &d\nu_t+2\int_{M}|\nabla(\tilde{y}_j^2-y_j^2)|^{2}d\nu_t\\&=2\int_{M}(\tilde{y}_j^2-y_j^2)\Box(\tilde{y}_j^2-y_j^2)d\nu_t \\ & \leq 4\int_{M}|1-|\nabla y_j|^2|\cdot |\widetilde{y}_j^2 - y_j^2|d\nu_t  \\ & \leq \sigma \int_M (\widetilde{y}_j^2-y_j^2)^2 d\nu_t + 4\sigma^{-1}\int_M (1-|\nabla y_j|^2)^2 d\nu_t \\
    &\leq \sigma \int_M (\widetilde{y}_j^2 - y_j^2)^2 d\nu_t +\sigma^{-1}\Psi(\delta|Y,\epsilon), 
\end{align*}
where for the last inequality we used Proposition \ref{prop:splittingproperties}\ref{eq:splittingproperty1}. The claim follows by taking $\sigma = \Psi(\delta|Y,\epsilon)^{\frac{1}{2}}$  and integrating with respect to time.
\end{proof}

Set $u:=\frac{1}{4\tau}\sum_{j=1}^k y_j^2$, so that $h=f'-u$. Set $u^\prime:=\frac{1}{4\tau}\sum_{j}\tilde{y}_j^2$ and $h^{\prime}:=f'-u^{\prime}$. It follows from Claim \ref{Claim:: strong_potential2} and $h-h'=u'-u$ that $h^{\prime}$ satisfies \eqref{eq : strong_potential1}, and so it remains to verify it is a strong $(k,\epsilon ,1)$-soliton potential. Identity \ref{eq:strongheateq1} is satisfied by construction. By Proposition \ref{prop:splittingproperties}\ref{eq:splittingproperty4}, 
\begin{equation} \label{eq:almostorthogonality} \int_{-4\epsilon^{-1}}^{-\frac{1}{4}\epsilon}\int_M |\langle 4\tau \nabla f',\nabla y_j\rangle -2y_j|^2d\nu_t dt \leq \Psi(\delta|\epsilon,Y).\end{equation}
Combining this with \eqref{eq : strong_potential1} and Proposition \ref{prop:splittingproperties}\ref{eq:splittingproperty1} yields
\begin{align*} 
   \int_{-4\epsilon^{-1}}^{-\frac{1}{4}\epsilon } \int_{M}|\langle 4\tau \nabla h^{\prime},\nabla y_{j}\rangle|^{2} d\nu_t &=   \int_{-4\epsilon^{-1}}^{-\frac{1}{4}\epsilon } \int_{M}|\langle 4\tau \nabla (h^{\prime}-h),\nabla y_{j}\rangle+\langle 4\tau \nabla h,\nabla y_{j}\rangle|^{2} d\nu_tdt \\ &\leq \int_{-4\epsilon^{-1}}^{-\frac{1}{4}\epsilon } \int_{M}|\langle 4\tau( \nabla f'- \frac{1}{2\tau}y_{j}\nabla y_{j}),\nabla y_{j}\rangle|^{2} d\nu_t dt + \Psi(\delta|\epsilon,Y)\\& \leq \int_{-4\epsilon^{-1}}^{-\frac{1}{4}\epsilon } \int_{M}|4\tau \langle  \nabla f',\nabla y_{j}\rangle - 2y_{j}+2y_{j}(1-|\nabla y_{j}|^{2})|^{2} d\nu_t dt +\Psi(\delta |\epsilon,Y)\\  & \leq \Psi(\delta | \epsilon, Y),
\end{align*}
which shows that $h^{\prime}$ satisfies \ref{eq:strongheateq2}. Next we verify \ref{eq:strongheateq3}. By Proposition \ref{prop:splittingproperties}\ref{eq:splittingproperty3},
\begin{align*} \int_{-2\epsilon^{-1}}^{-\frac{1}{2}\epsilon} \int_M \tau \left| \nabla^2 u - \frac{1}{2\tau}\sum_{j=1}^k dy_j \otimes dy_j \right|^2 d\nu_t dt \leq C\sum_{j=1}^k \int_{-2\epsilon^{-1}}^{-\frac{1}{2}\epsilon} \int_M |y_j|^2 \cdot |\nabla^2 y_j|^2 d\nu_t dt \leq \Psi(\delta|\epsilon,Y).
\end{align*}
By this and Definition \ref{def : strong_potential}\ref{iden : strong_potential13}, it therefore suffices to show that
\begin{equation} \label{eq:needtoshowforiii} \int_{-2\epsilon^{-1}}^{-\frac{1}{2}\epsilon} \int_M \tau \left|\nabla^2(u-u') \right|^2 d\nu_t dt \leq \Psi(\delta|\epsilon,Y). \end{equation}
Let $\eta \in C_c^{\infty}([-4\epsilon^{-1},-\frac{1}{4}\epsilon])$ be a cutoff function satisfying $\eta|_{[-2\epsilon^{-1},-\frac{1}{2}\epsilon]}\equiv 1$ and $\tau |\eta'|\leq 10$. From $\square(\tau u')=-\frac{k}{2}$ and $\square(\tau u)= -\frac{1}{2}\sum_{j=1}^k |\nabla y_j|^2$, we obtain
\begin{align*} \frac{d}{dt} \int_M \tau^2 \left| \nabla (u-u') \right|^2 d\nu_t +2\tau^2 \int_M |\nabla^2(u-u')|^2 d\nu_t =&  2 \int_M \tau \langle \nabla(u-u'),\nabla \square (\tau u-\tau u')\rangle d\nu_t \\
\leq & \int_M |\nabla (u-u')|^2 d\nu_t + C(\epsilon) \int_M |\nabla^2 y_i|^2 |\nabla y_i|^2 d\nu_t.
\end{align*}
We integrate this in time against $\eta$, and apply Claim \ref{Claim:: strong_potential2} and \ref{prop:splittingproperties}\ref{eq:splittingproperty3} to obtain
\begin{align*} \int_{-2\epsilon^{-1}}^{-\frac{1}{2}\epsilon}\tau^2 \int_M |\nabla^2(u-u')|^2 d\nu_t dt \leq & C(\epsilon)\int_{-4\epsilon^{-1}}^{-\frac{1}{4}\epsilon} (1+|\eta'|)\int_M \left( |\nabla (u-u')|^2 + \sum_{j=1}^k |\nabla^2 y_j|^2 |\nabla y_j|^2 \right) d\nu_t dt \\ \leq & \Psi(\delta|\epsilon,Y),
\end{align*}
so \ref{eq:strongheateq3} holds. To prove \ref{eq:strongheateq4}, we first set
\begin{align*}
    w^{\prime}&:= \tau(R+2\Delta h^{\prime}-|\nabla h^{\prime}|^{2})+h^{\prime}-(n-k), \\ w&:= \tau(R+2\Delta f'-|\nabla f'|^{2})+f'-n,\\
          v'&:=\tau(2\Delta_{f'} u^{\prime}+|\nabla u^{\prime}|^{2})+u^{\prime} -k,\\
          v&:=\tau(2\Delta_{f'} u+|\nabla u|^2)+u-k,
\end{align*}
so that $w'=w-v'$.

\begin{Claim} \label{Claim : strong_potential3}
     \begin{align*}
            \int_{-2\epsilon^{-1}}^{-\frac{1}{2}\epsilon }\int_{M} |v'|\,d\nu_t  dt\leq \Psi(\delta|\epsilon,Y).
        \end{align*} 
    \end{Claim}
\begin{proof}
By \eqref{eq:needtoshowforiii} and Claim \ref{Claim:: strong_potential2}, we have
\begin{align*} \int_{-2\epsilon^{-1}}^{-\frac{1}{2}\epsilon} \int_M |v-v'|d\nu_t dt \leq & C\int_{-2\epsilon^{-1}}^{-\frac{1}{2}\epsilon} \int_M \tau \left( |\nabla^2 (u-u')|+|\nabla (u-u')|\cdot |\nabla f'|+ |u-u'| \right) d\nu_t dt \\
 & + C(\epsilon)\int_{-2\epsilon^{-1}}^{-\frac{1}{2}\epsilon} \int_M |\nabla (u-u')|\cdot (|\nabla(u-u')|+|\nabla u|)d\nu_t dt\\
 \leq & \Psi(\delta|\epsilon,Y),
\end{align*}
where for the last inequality, we also used Proposition \ref{prop:splittingproperties}\ref{eq:splittingproperty1} and \ref{prop:boundforf}. 
By Proposition \ref{prop:splittingproperties}\ref{eq:splittingproperty1},\ref{eq:splittingproperty3},\ref{eq:splittingproperty4}, we have
\begin{align*} \int_{-2\epsilon^{-1}}^{-\frac{1}{2}\epsilon} \int_M |v| d\nu_t dt \leq & C(\epsilon) \sum_{j=1}^k \int_{-2\epsilon^{-1}}^{-\frac{1}{2}\epsilon} \int_M \left( \left| |\nabla y_j|^2 -1 \right| + |y_j| \cdot |\Delta y_j| +|y_j|\cdot  |2y_j-\langle \nabla y_j,\nabla (4\tau f')\rangle | \right) d\nu_t dt \\
&+C(\epsilon)\sum_{i,j=1}^k \int_{-2\epsilon^{-1}}^{-\frac{1}{2}\epsilon} \int_M|y_i y_j|\cdot |\langle \nabla y_i, \nabla y_j \rangle -\delta_{ij}|d\nu_tdt\\
\leq & \Psi(\delta|\epsilon,Y),
\end{align*}
so the claim follows. 
\end{proof}

\begin{Claim} \label{Claim : strong_potential4}
\begin{align*} \int_{-2\epsilon^{-1}}^{-\frac{1}{2}\epsilon} \int_M |\square (\tau v')| d\nu_t dt\leq \Psi(\delta|\epsilon,Y).
\end{align*}
\end{Claim}

\begin{proof}
     A standard computation gives
\begin{align*}
\square (\tau v') =& 4\tau^2 \left\langle Rc + \nabla^2 f'-\frac{1}{2\tau}g,\nabla^2 u'\right\rangle -\frac{1}{4}\left( |\nabla^2 (4\tau u')|^2 -2k \right).
\end{align*}
Using \eqref{eq:needtoshowforiii}, Claim \ref{Claim:: strong_potential2}, and Proposition \ref{prop:splittingproperties}\ref{eq:splittingproperty3}, we can estimate
\begin{align*} \int_{-2\epsilon^{-1}}^{-\frac{1}{2}\epsilon} \int_M  |\nabla^2 (4\tau u')-2\sum_{j=1}^k dy_j \otimes dy_j|^2 d\nu_t dt \hspace{-50 mm} \\
\leq & C(\epsilon)\int_{-2\epsilon^{-1}}^{-\frac{1}{2}\epsilon}\int_M |\nabla^2 (u'-u)|^2 d\nu_t dt + C(\epsilon)\sum_{j=1}^k \int_{-2\epsilon^{-1}}^{-\frac{1}{2}\epsilon} \int_M y_j^2 |\nabla^2 y_j|^2 d\nu_t dt\\
\leq & \Psi(\delta|\epsilon,Y).
\end{align*}
By Proposition \ref{prop:splittingproperties}\ref{eq:splittingproperty1} we have 
\begin{align*} \int_{-2\epsilon^{-1}}^{-\frac{1}{2}\epsilon} \int_M \left| \left| 2\sum_{j=1}^k dy_j \otimes dy_j \right|^2 -2k \right| d\nu_t dt \leq C(\epsilon,Y)\sum_{i,j=1}^k \int_{-2\epsilon^{-1}}^{-\frac{1}{2}\epsilon}\int_M |\langle \nabla y_i,\nabla y_j\rangle^2 -1|d\nu_t dt\leq \Psi(\delta|Y,\epsilon),
\end{align*}
so the claim follows by combining expressions.
\end{proof}

Fix a cutoff function $\xi \in C^{\infty}(\mathbb{R})$ satisfying $|\xi'|\leq 10$, $\xi|_{(-\infty,-2\epsilon^{-1}]} \equiv 0$, $\xi|_{[-\epsilon^{-1},-\epsilon]} \equiv 1$. For $t_1 \in [-\epsilon^{-1},-\epsilon]$, integrating 
$$\frac{d}{dt}\int_M |\tau v'|d\nu_t \leq \int_M |\square (\tau v')|d\nu_t$$
against $\xi$ on the time interval $[-2\epsilon^{-1},t_1]$ yields
\begin{align*} \int_M |\tau v'|d\nu_{t_1} \leq &\int_{-2\epsilon^{-1}}^{-\epsilon^{-1}} |\xi'| \tau \int_M |v'|d\nu_t dt + \int_{-2\epsilon^{-1}}^{-\frac{1}{2}\epsilon} \int_M |\square (\tau v')|d\nu_t dt\\
\leq & \Psi(\delta|\epsilon,Y),
\end{align*}
where we used Claim \ref{Claim : strong_potential3} and Claim \ref{Claim : strong_potential4}. Combining this with Definition \ref{def : strong_potential}\ref{iden : strong_potential14} yields \ref{eq:strongheateq4}. Moreover, \ref{eq:strongheateq5} is an immediate consequence of \ref{eq:strongheateq3}, \ref{eq:strongheateq4}, and Proposition \ref{prop:splittingproperties}\ref{eq:splittingproperty1}.

Finally, we verify \ref{eq:strongheateq6}. Note that 
\begin{align*}
    \frac{d}{dt}\left(\int_M \tau( h^\prime -\frac{n-k}{2})d\nu_t-\tau W\right)=0,
\end{align*}
and by Proposition \ref{prop:splittingproperties}\ref{eq:splittingproperty2} and  Definition \ref{def : strong_potential}\ref{iden : strong_potential15}  we have
\begin{align*}
    \left| \int_M (h^\prime -\frac{n-k}{2})\,d\nu_{t^\ast}-W \right|&\leq \left|\int_M \left( f'-\frac{n}{2} \right)\,d\nu_{t^\ast }-W \right| + \frac{1}{4\tau} \sum_{j=1}^k \left| \int_M (y_j^2 - 2\tau)d\nu_t \right| \leq \Psi(\delta|\epsilon,Y).    
\end{align*}
Hence, we may add a very small constant to $h^\prime $ so that \ref{eq:strongheateq6} is satisfied, without affecting \ref{eq:strongheateq1}-\ref{eq:strongheateq5}.
\end{proof}

The following $L^p$-estimates will be necessary for applying Lemma \ref{lem: uniform_cont}.

\begin{Lemma} \label{lem : Lp_estimates}
    For any $\epsilon>0$, $p\in [1,\infty)$, and $Y <\infty$, the following holds if $\delta \leq \overline{\delta}(\epsilon,Y,p)$. Let $(M^n,(g_{t})_{t\in I })$ be a compact Ricci flow, $r>0$, $(x_0,t_0)\in M \times I $ satisfy $[t_0-\delta^{-1}r^2,t_0]\subseteq I$ and $W\coloneqq \mathcal{N}_{x_0,t_0}(r^2)\geq -Y$. If $h\in C^\infty(M\times [t_0-\delta^{-1}r^2,t_0-\delta r^2])$ is a strong $(k,\delta,r)$-soliton potential and $q\coloneqq 4\tau(h-W)$, then
    \begin{equation} \sup_{t\in [t_0-\epsilon^{-1}r^2,t_0-\epsilon r^2]}\int_M (1+|q|+|\nabla q|)^p\,d\nu_{x_0,t_0;t}\leq C(Y,\epsilon,p). \label{iden : Lp_estimates} \end{equation} \end{Lemma}
\begin{proof}
    The proof is identical to that of Proposition \ref{prop:boundforf}.
    
 \end{proof}

%% file: Alt_Tech_Lemma.tex
\section{Uniform Convergence of Almost-Soliton Potentials \label{section:bubble}}

The following Lemma asserts that if the limit is static and uniformally non-collapsed, any limiting conjugate heat flow must be given by the conjugate heat kernel based at a point. 

\begin{Lemma} \label{lem:staticpoints}  Suppose $(M_{i}^{n},(g_{i,t})_{[-\delta_{i}^{-1},0]})$ is a sequence
of closed Ricci flows and $x_{0,i}\in M_i$ are such that $\mathcal{N}_{x_{0,i},0}(\delta_i^{-1})\geq-Y$
and
\[
(M_{i},(g_{i,t})_{t\in[-\delta_{i}^{-1},0]},(\nu_{x_{0,i},0;t})_{t\in[-\delta_{i}^{-1},0]})\xrightarrow[i\to\infty]{\mathbb{F},\mathfrak{C}}(\mathcal{X},(\nu_{t})_{t\in(-\infty,0)})
\]
uniformly on compact time intervals for static metric flow $\mathcal{X}$. If $(\mu_t)_{t\in (-\infty,t_0)}$ is a conjugate heat flow on $\mathcal{X}$ satisfying
	$$\lim_{t \nearrow t_0} \operatorname{Var}_{\mathcal{X}_t}(\mu_t) =0$$
    for some $t_0 \in (-\infty,0)$, then there exists $x_0\in \mathcal{X}_{ t_0}$ such that $\mu_t = \nu_{x_0;t}$ for all $t\in (-\infty,t_0)$.
\end{Lemma}
\begin{proof} Let $(X ,d,(\nu_{x;t}')_{x\in X;t\in (-\infty,0)})$ be the static model for $\mathcal{X}$, as in \cite[Definition 3.54]{Bam3}. Note that $\mathcal{N}_{x}(\tau)\geq -Y$ for $x\in \mathcal{X}$ and $\tau> 0$. For any $t<t_0$, let $z_t \in \mathcal{X}_{t}$ be an $H_n$-center of $\mu_t$. For $t_2<t_1<t_0$, we have  using \cite[Claim $22.9$(a)]{Bam3} that
	\begin{align*}
		d_{t_2}(z_{t_1},z_{t_2}) & \leq 2\sqrt{H_n(t_0-t_2)} + d_{W_1}^{\mathcal{X}_{t_2}}(\delta_{z_{t_1}(t_2)},\nu_{z_{t_1};t_2}) \leq C(Y)\sqrt{t_0-t_2}. 	
	\end{align*}	
	It follows that $(z_{t_i})$ corresponds to a Cauchy sequence in $X$, so converges to some $x \in X$ such that $(x,t_0)$ corresponds to a point $x_0 \in \mathcal{X}_{t_0}$. Because the natural topology of a static flow coincides with the product topology on $X\times (-\infty,0)$, we then also have 
	$$d_{W_1}^{\mathcal{X}_{t}}(\nu_{x_0;t},\mu_t) \leq \liminf_{i\to \infty} \left( d_{W_1}^{\mathcal{X}_t}(\nu_{x_0;t},\nu_{z_{t_i};t})+ d_{W_1}^{\mathcal{X}_t}(\nu_{z_{t_i};t},\mu_t) \right)=0$$
	for any $t>0$, hence $\nu_{x_0;t}=\mu_t$. 
\end{proof}

The following Lemma is an application of \cite[Theorem $6.49$]{Bam2} in the case that the basepoints are almost selfsimilar. Here we only assume closeness of the heat kernels at an earlier time.

\begin{Lemma} \label{lem: staticpoints2}

Let $(M_{i}^{n},(g_{i,t})_{[-\delta_{i}^{-1},0]})$ be a sequence
of closed Ricci flows, and consider $x_{0,i}\in M_i$ which satisfy $\mathcal{N}_{x_{0,i},0}(1)\geq-Y$ and which $\mathbb{F}$-converge as in \eqref{eq:Fconvergence} uniformly on compact time intervals for some metric soliton $\mathcal{X}$,
and some correspondence $\mathfrak{C}$.  If $(y_i,s_i)\in M_i\times [-\delta_i^{-1},0]$ and $\widetilde{s}_i \in [-\delta_i^{-1},s_i]$ satisfy
$$\limsup_{i \to \infty} d_{W^1}^{g_{i,\tilde{s}_i}}( \nu_{y_i,s_i; \tilde{s}_i}, \nu_{x_{i},0; \tilde{s}_i}) <\infty, \qquad \limsup_{i\to \infty} (|s_i|^{-1}+|\widetilde{s}_i|)<\infty,$$
then after passing to a subsequence, there exists a conjugate heat flow $(\mu_s) _{s\leq s_\infty}$, with $s_\infty=\lim_{i\to\infty} s_i \in (-\infty,0)$ such that \begin{align*}
			(\nu_{y_i,s_i;t})_{t\in [-\delta_i^{-1},s_i]} \xrightarrow[i\to\infty]{\mathfrak{C}} (\mu_t)_{t\in (-\infty,s_{\infty}]}, \qquad  \lim_{s\nearrow s_\infty}\operatorname{Var}_{\mathcal{X}_s}(\mu _s)=0.
            \end{align*} 
\end{Lemma}

\begin{proof}
	 
      Arguing as in \cite[Theorem 6.49]{Bam2}, it suffices to prove the following claim.
	\begin{Claim} \label{Claim : staticpoints2_1}
        For each $s\in (-\infty,s_{\infty})$, there exists a probability measure $\mu_s\in \mathcal{P}(\mathcal{X}_{s})$ such that $$\nu_{y_i,s_i;s}\xrightarrow[i\to \infty]{\mathfrak{C}} \mu_s$$ strictly. 
	\end{Claim}
	\begin{proof}
		Let $\mathfrak{C}$ be as in \eqref{eq:correspondence}. Let $z_{i,s}$ be $H_n$-centers of $\nu_{y_i,s_i;s}$. Let $z_{\infty,s}$ be a $H_n$-center of $\nu_{x_\infty;s}$. By  \cite[Proposition $9.1$]{Bam3}, there exists  $C(s,Y) <\infty$ such that  $$d_{W^1}^{g_{i,s}}(\nu_{y_i,s_i;s},\nu_{x_{i},0;s})\leq C(s ,Y).$$ Note that we have \begin{align*}
			d_{Z_s}(\varphi_{i,s}(z_{i,s}),\varphi_{\infty, s}(z_{\infty,s}))&\leq d_{W^1}^{Z_s}((\varphi_{i,s})_\ast \delta_{z_i,s},(\varphi_{i,s})_\ast \nu_{y_i,s_i;s})+d_{W^1}^{Z_s}((\varphi_{\infty,s})_\ast \delta_{z_\infty,s},(\varphi_{\infty,s})_\ast \nu_{x_\infty;s})\\ &+d^{Z_s}_{W^1}( (\varphi_{i,s})_\ast \nu_{x_i,0;s},(\varphi_{\infty,s})_\ast \nu_{x_\infty;s })+d_{W^1}^{Z_s}((\varphi_{i,s})_\ast \nu_{y_i,s_i;s},(\varphi_{i,s})_\ast \nu_{x_i,0;s}) \\ & \leq C(s,Y).
		\end{align*} 
		Therefore, for all $B<\infty$, we may choose a large constant $A(s,Y,B)<\infty$ such that $B_{Z_s}(z_{\infty,s};A)\supset B_{Z_s}(z_{i,s};B) $ for large $i\in \mathbb{N}$. It follows by \cite[Lemma 3.37]{Bam2} that the sequence $(\varphi_{i,s})_\ast \nu_{y_i,s_i;s}$ is tight, and therefore, we can argue as in the proof of \cite[Claim 6.54] {Bam2} to get the claim.
	\end{proof}
\end{proof}

We now define a notion of locally uniform convergence of functions within a correspondence.

\begin{Definition} Suppose $(\mathcal{X}_i,(\mu_t^i)_{t\in I_i})_{i\in \mathbb{N}}$ is a sequence of metric flow pairs converging within some correspondence uniformly over an interval $J\subseteq \cap_i I_i$. Given an increasing sequence of intervals $J_i \subseteq I_i$, and given functions $\phi_i \in C(\mathcal{X}_i|_{J_i})$ and $\phi \in C(\mathcal{X}|_J)$, we say that $\phi_i$ converge locally uniformly within $\mathfrak{C}$ over $J$, written 
$$\lim_{i \to \infty}^{\mathfrak{C},J} \phi_i =\phi,$$ 
if $\cup_i J_i \supseteq J$ and for any sequence $x_i \in \mathcal{X}_i$ such that $x_i \xrightarrow[i\to \infty]{\mathfrak{C}} x \in \mathcal{X}|_J$, we have 
$$\lim_{i\to \infty} \phi_i(x_i) = \phi(x).$$
If instead we have $\limsup_{i\to \infty} \phi_i(x_i)\leq\phi(x)$ for any such sequence $x_i$, we then write
$$\limsup_{i\to \infty}^{\mathfrak{C},J}\phi_i \leq \phi.$$
\end{Definition}

The following lemma gives a sufficient condition for which a sequence of solutions to the heat equation converge uniformly with respect to a correspondence.

\begin{Lemma} \label{lem: uniform_cont}
Suppose $(M_{i}^{n},(g_{i,t})_{[-\delta_{i}^{-1},0]})$ is a sequence
of closed Ricci flows and $x_{0,i}\in M_i$ are such that $\mathcal{N}_{x_{0,i},0}(1)\geq-Y$
and the $\mathbb{F}$-convergence \eqref{eq:Fconvergence} holds uniformly on compact time intervals for some metric flow $\mathcal{X}$
and some correspondence $\mathfrak{C}$. Let $\psi_i$ be as in Theorem \ref{bamconvergence}, and suppose
$u_{i}\in C^{\infty}(M_{i}\times[-\delta_{i}^{-1},\delta_{i}])$ are
solutions of the heat equation satisfying the following for some $p\in (1,\infty)$:
\begin{enumerate}
    \item  $\psi_{i}^{\ast}u_{i}\to u$ in $C_{\text{loc}}^{\infty}(\mathcal{R})$
for some $u\in C^{\infty}(\mathcal{R})\cap C(\mathcal{X})$ \label{iden : uniform_cont1},

\item for any $t\in(-\infty,0)$, we have $\limsup_{i\to\infty}\int_{M_{i}}|u_{i}|^{p}d\nu_{x_{0,i},0;t}<\infty$ \label{iden : uniform_cont2},

\item for any $t_{1}<t_{2}<0$ and $x\in\mathcal{X}_{t_{2}}$,
we have $u_{t_{2}}(x)=\int_{\mathcal{R}_{t_{1}}}u_{t_{1}}d\nu_{x,t_{1}}.$ \label{iden : uniform_cont3}
\end{enumerate}

\noindent Then $\lim_{i\to \infty}^{\mathfrak{C},(-\infty,0)}u_i =u$.
\end{Lemma}

\begin{Remark} \label{Remark : uniform_cont1}
If $\mathcal{X}$ is a metric soliton, it can be shown that \ref{iden : uniform_cont3} is superfluous, but the proof is simpler if we make
this assumption.
\end{Remark}

\begin{proof}
Suppose $x\in\mathcal{X}_{t}$, and fix $t'\in(-\infty,t)$, $\epsilon>0$.
Set $t^{\ast}:=\frac{1}{2}(t+t')$, and write $\mathfrak{C}$ as in \eqref{eq:correspondence}. By definition, we then have
\[
\lim_{i\to\infty}d_{W_{1}}^{Z_{t^{\ast}}}\left((\varphi_{t^{\ast}}^{i})_{\ast}\nu_{x_{i},t_{i};t^{\ast}},(\varphi_{t^{\ast}})_{\ast}\nu_{x;t^{\ast}}\right)<\infty,
\]
so that from 
\begin{align*}
d_{W_{1}}^{g_{i,t^{\ast}}}(\nu_{x_{i},t_{i};t^{\ast}},\nu_{x_{0,i},0;t^{\ast}})= & d_{W_{1}}^{Z_{t^{\ast}}}\left((\varphi_{t^{\ast}}^{i})_{\ast}\nu_{x_{i},t_{i};t^{\ast}},(\varphi_{t^{\ast}})_{\ast}\nu_{x_{0,i},0;t^{\ast}}\right)\\
\leq & d_{W_{1}}^{Z_{t^{\ast}}}\left((\varphi_{t^{\ast}}^{i})_{\ast}\nu_{x_{i},t_{i};t^{\ast}},(\varphi_{t^{\ast}})_{\ast}\nu_{x;t^{\ast}}\right)+d_{W_{1}}^{\mathcal{X}_{t^{\ast}}}(\nu_{x;t^{\ast}},\mu_{t^{\ast}})+d_{W_{1}}^{\mathcal{X}_{t^{\ast}}}\left((\varphi_{t^{\ast}})_{\ast}\mu_{t^{\ast}},(\varphi_{t^{\ast}}^{i})_{\ast}\nu_{x_{0,i},0;t^{\ast}}\right),
\end{align*}
we have
\begin{align*}
D:= & \limsup_{i\to\infty}d_{W_{1}}^{g_{i,t^{\ast}}}(\nu_{x_{i},t_{i};t^{\ast}},\nu_{x_{0,i},0;t^{\ast}})\leq d_{W_{1}}^{\mathcal{X}_{t^{\ast}}}(\nu_{x;t^{\ast}},\mu_{t^{\ast}})<\infty.
\end{align*}
For any $\alpha>0$, \cite[Propoisition $8.1$]{Bam3} then gives
\[
d\nu_{x_{i},t_{i};t'}\leq C(Y,D,t',t,\alpha)e^{\alpha f_{i}}d\nu_{x_{0,i},0;t'},
\]
where $f_{i}$ are defined by $d\nu_{x_{0,i},0;t}=(4\pi|t|)^{-\frac{n}{2}}e^{-f_{i}}dg_{i,t}$. 

Let $L\subseteq\mathcal{R}_{t'}$ be a compact subset to be determined. Using \ref{iden : uniform_cont1}-\ref{iden : uniform_cont2} we have
\begin{align*}
\limsup_{i\to\infty}\left|\int_{M_{i}}u_{i}d\nu_{x_{i},t_{i};t'}-\int_{\mathcal{R}_{t'}}u_{t'}d\nu_{x;t'}\right| \leq & \limsup_{i\to \infty} \int_{M_i \setminus \psi_{i,t'}(L)} |u_i| d\nu_{x_i,t_i;t'} + \int_{\mathcal{R}_{t'}\backslash L}|u_{t'}| d\nu_{x;t} \\ \leq & \limsup_{i\to \infty} \left( \int_{M_i} |u_i|^p d\nu_{x_i,t_i;t} \right)^{\frac{1}{p}} \nu_{x_i,t_i;t'}(M_i \setminus \psi_{i,t'}(L))^{\frac{p-1}{p}}\\& + \left( \int_{\mathcal{R}} |u|^p d\nu_{x;t}\right)^{\frac{1}{p}} \nu_{x;t}(\mathcal{R}_{t'} \setminus L)^{\frac{p-1}{p}} \\
\leq & 2\limsup_{i\to \infty} \left( \int_{M_i} |u_i|^p d\nu_{x_i,t_i;t} \right)^{\frac{1}{p}} \nu_{x;t}(\mathcal{R}_{t'}\setminus L)^{\frac{p-1}{p}}, 
\end{align*}
where we used the fact (see \cite{Bam3}[Theorem 9.31(e)]) that $\psi_{i,t'}^{\ast}K(x_i,t_i;\cdot,t')\to K(x;\cdot)$ on $\mathcal{R}_{t'}$ since $(x_i,t_i)\xrightarrow[i \to \infty]{\mathfrak{C}} x$. By choosing $L=L(x,t',\epsilon)$ appropriately, we can ensure that
the right hand side is at most $\frac{\epsilon}{3}$. The result then follows from \ref{iden : uniform_cont3} and the fact that $u_i$ solves the heat equation.

\end{proof}

The following corollary summarizes the important conclusion of the preceding results. 

\begin{Corollary} \label{cor : uni_cont}
Let $(M_{i}^{n},(g_{i,t})_{t\in [-\epsilon_i^{-1},0]})$ be a sequence
of closed Ricci flows, and assume $(x_{0,i},0)\in M_i\times \{0\}$ are strongly $(n-4,\epsilon_i,1)$-selfsimilar and $(\epsilon_i,1)$-static for some $\epsilon_i\searrow 0$, and that $W_i\coloneqq \mathcal{N}_{x_{0,i},0}(1)\geq-Y$. Let $q_i:=4\tau(h_i-W_i)$ where $h_i$ are $(n-4,\epsilon_i,1)$-soliton potentials, and $y_i=(y_i^1,...,y_i^{n-4})$ are strong splitting maps. After passing to a subsequence, we then have $\mathbb{F}$-convergence as in \eqref{eq:Fconvergence} where $\mathcal{X}$ is a static metric flow modeled on a flat cone of the form $C(\mathbb{S}^3/\Gamma) \times \mathbb{R}^{n-4}$ with vertex $(x_\ast,0^k)$. In the notation of Theorem \ref{bamconvergence}, we have 
$$\psi_i^{\ast}q_i \to q \qquad \psi_i^{\ast}y_i^j \to y^j$$ in $C_{\operatorname{loc}}^{\infty}(\mathcal{R})$ as $i\to \infty$, where $ y^j$ is the projection onto the $j$th Euclidean factor and $q \in C^\infty(\mathcal{R})\cap C(\mathcal{R})$ is given by the square of the radial coordinate function. Moreover, we have $$\lim_{i\to\infty }^{\mathfrak{C},(-\infty ,0)}h_i=h\qquad \lim_{i\to\infty }^{\mathfrak{C},(-\infty ,0)}y^j_i=y^j\qquad \limsup_{i\to\infty}^{\mathfrak{C},(-\infty ,0)}|\nabla y_i^j|\leq 1.$$ In particular, for any $(x_i^\prime ,t_i^\prime)\in P^\ast(x_{0,i},0;D)$ where $D<\infty $ and $t_i^\prime \in [-D,-D^{-1}]$, we can pass to a subsequence so that $(x_i^\prime, t_i^\prime)\xrightarrow[i\to\infty ]{\mathfrak{C}} x_\infty^\prime$ for some $x_\infty^\prime \in \mathcal{X}$ and the following statements hold:
 \begin{enumerate}
     \item $\lim_{i\to\infty }h_i(x_i^\prime,t_i^\prime)=h(x_\infty ^\prime) $\label{iden : uni_cont1}
     \item  $\lim_{i\to\infty }y^j_i(x_i^\prime,t_i^\prime)=y^j(x_\infty ^\prime)$\label{iden : uni_cont2}
     \item $\lim_{i\to\infty }|\nabla y_{i}^j|(x_i^\prime , t_i^\prime) \leq 1$.  \label{iden : uni_cont3}
 \end{enumerate}
\end{Corollary}

\begin{proof}
The $\mathbb{F}$-convergence \eqref{eq:Fconvergence} follows from Theorem \ref{bamconvergence}. The statements regarding local smooth convergence of $q_i,y_i^j$ and the fact that $\mathcal{X}$ is a static metric flow modeled on $C(\mathbb{S}^3/\Gamma)\times \mathbb{R}^{n-4}$ follow exactly as in \cite[Theorem $15.50$, Theorem $15.69$, Theorem $15.80$]{Bam3}. 

\begin{Claim} \label{claim:hypothesisiiiholds} $q+8t$ and $y^j$ satisfy hypothesis \ref{iden : uniform_cont3} of Lemma \ref{lem: uniform_cont}.
\end{Claim}
\begin{proof} Define $\chi_r(x,t):= \chi(r^{-1}\sqrt{q})$, where $\chi \in C_c^{\infty}([0,2])$ satisfies $0\leq \chi \leq 1$ and $\chi|_{[0,1]}\equiv 1$. Fix $t_0<t_1 <0$ and $x_0 \in \mathcal{X}_{t_0}$. Let $u$ be one of $y^j,q$. Then
\begin{align*} \frac{d}{dt} \int_{\mathcal{R}_{t_0}} u d\nu_{x_0(t);t_0} &= \int_{\mathcal{R}_{t_0}} u(y) \partial_{\mathfrak{t}}|_{x_0(t)} K(\cdot;y) dg_{t_0}(y) = \int_{\mathcal{R}_{t_0}} u(y)\partial_t|_{y} K(x_0(t);y)dg_{t_0}(y) \\& = -\int_{\mathcal{R}_{t_0}} u(y)\Delta_y K(x_0(t);y)dg_{t_0}(y).
\end{align*}
For any $r>0$, $\Delta q = -8$, $|\Delta \chi_r|\leq Cr^{-1}$, and $\operatorname{supp}(\chi '\circ q) \subseteq \{\sqrt{q} \geq r\}$ together imply
\begin{align*} \left|  \int_{\mathcal{R}_{t_0}} (q \chi_r)(y)\Delta_y K(x_0(t);y)dg_{t_0}(y) -8\right|\leq & \left| \int_{\mathcal{R}_{t_0}}(2r^{-1}(\chi'\circ q)|\nabla q|^2 + q\Delta \chi_r)(y)K(x_0(t);y)dg_{t_0}(y) \right| \\
&+8\int_{\mathcal{R}_{t_0}} |1-\chi_r|d\nu_{x_0,t_0}\\
 \leq & \Psi(r^{-1}|x_0).
\end{align*}
Similarly, from $\Delta y^j=0$, we have
$$\left| \int_{\mathcal{R}_{t_0}} (y^j \chi_r)(y)\Delta_y K(x_0(t);y)dg_{t_0}(y) \right| \leq \Psi(r^{-1}|x_0).$$
By integrating in time from $t=t_0$ to $t=t_1$, taking $r\to \infty$, and setting $x_1 :=x_0(t)$, we then have
$$q(x_1)= \int_{\mathcal{R}_{t_0}}q d\nu_{x_1;t_0}-8(t_1-t_0), \qquad y^j(x_1) = \int_{\mathcal{R}_{t_0}} y^j d\nu_{x_1;t_0}.$$
\end{proof}

 Note that $\Box (q_i-8\tau)=0$ and $\Box y^j_i=0$ by Definition \ref{def : strong_potential2}\ref{eq:strongheateq1} and Definition \ref{def : strong_splitting}\ref{iden : strong_splitting1}. The  statements regarding uniform convergence of $h_i$ and $y_i^j$ within a correspondence then follow from Claim \ref{claim:hypothesisiiiholds}, Lemma \ref{lem: uniform_cont}, and Lemma \ref{lem : Lp_estimates}. Because
$$|\nabla y_i^j(x_i',t_i')|\leq \int_{\mathcal{R}_{t_0}} |\nabla y_i^j| d\nu_{x_i',t_i';t_0},$$
the fact that $\limsup_{i\to \infty}^{\mathfrak{C},(-\infty,0)}|\nabla y_i^j|\leq 1$ follows from Proposition \ref{prop:splittingproperties}\ref{eq:splittingproperty1} and the proof of Lemma \ref{lem: uniform_cont}.
The fact that $(x_i^\prime ,t_i^\prime)\xrightarrow[i\to\infty]{\mathfrak{C}}$ follows from Lemma \ref{lem:staticpoints} and Lemma \ref{lem: staticpoints2},  and so \ref{iden : uni_cont1}-\ref{iden : uni_cont3} follow by definition. 
\end{proof}

\begin{Lemma} \label{lem:bubbles1}
	For any $Y,\Lambda<\infty$ there exists $c=c(Y)>0$ such that the following statement whenever $\delta\leq \overline{\delta}(Y,\Lambda)$. Let $(M^n,(g_t)_{t\in I})$ be a closed Ricci flow, and suppose $(x_0,t_0)\in M\times I$ and $r>0$ satisfy $W:= \mathcal{N}_{x_0,t_0}(r^2)\geq -Y$. If $(x_0,t_0)$ is strongly $(n-4,\delta,r)$-selfsimilar and $(\delta,r)$-static, with strong $(n-4,\delta,r)$ soliton potential $h$, then $q:=4\tau(h-W) $ satisfies the following for every $(x,t)\in P^\ast (x_0,t_0;\Lambda r) \cap (M \times [t_0-\Lambda^2 r^2,t_0-\Lambda^{-2}r^2])$:
	\begin{enumerate}
		\item $q_t(x)\geq -\Lambda^{-2}r^2$ \label{iden:bubbles11},
		\item if $q_t(x)\geq \lambda^2 r^2$ for some $\lambda \in [\Lambda^{-1},\Lambda]$, then $r_{Rm}(x,t) \geq c\lambda r$. \label{iden:bubbles12}
	\end{enumerate}
\end{Lemma}

\begin{proof}
    By means of parabolic rescaling and application of a time-shift, we may assume that $r=1$ and $t_0=0$. Suppose that either \ref{iden:bubbles11} or \ref{iden:bubbles12} were false, so there is a sequence $\delta_i\searrow 0$ together with Ricci flows $(M^n_i, (g_{i,t})_{t\in [-\delta_i^{-1},0]})$, $(x_{0,i},0)\in M_i\times \{0\}$ which are strongly $(n-4,\delta_i,1)$-selfsimilar, $(\delta_i,1)$-static, $\mathcal{N}_{x_{0,i},0}(1)\geq -Y$, but such that for each $i\in \mathbb{N}$, either \ref{iden:bubbles11} or \ref{iden:bubbles12} fails.  That is, there is a sequence $(x^\prime_{i},t_i^\prime)\in P^\ast(x_{0,i},0;\Lambda)$, $t_i^\prime \in [-\Lambda,-\Lambda^{-1}]$, for which either \ref{iden:bubbles11} or \ref{iden:bubbles12} fails. By Corollary \ref{cor : uni_cont}, we can pass to a subsequence so that \eqref{eq:Fconvergence} holds for some static metric flow $\mathcal{X}$ modeled on a  flat cone $C(\mathbb{S}^3/\Gamma)\times \mathbb{R}^{n-4}$ with vertex $(0^{n-4},x_\ast)$. Moreover, we have $(x_i^\prime, t_i^\prime)\xrightarrow[i\to\infty ]{\mathfrak{C}} x_\infty^\prime$ for some $x_\infty^\prime \in \mathcal{X}$, and
\begin{align*}
    \lim_{i\to\infty }q_{i,t_i^\prime }(x_i^\prime)=q(x_\infty^\prime)
\end{align*}
where $q$ is the square of the radial coordinate function. Suppose first that \ref{iden:bubbles11} fails infinitely often. Then, after passing to a subsequence, we must have $$0\leq q(x_\infty^\prime)=\lim_{i\to\infty }q_{i,t_i^\prime }(x_i^\prime)<-\Lambda^{-2} ,$$
which is a contradiction. If \ref{iden:bubbles12} fails infinitely often for some $\lambda_i \in [\Lambda^{-1},\Lambda]$, then after passing to a subsequence, we have $\lambda_i \to \lambda \in [\Lambda^{-1},\Lambda]$, and from \cite[Lemma $15.16$]{Bam3} we have that 
\begin{align*}
    q(x_\infty^\prime)=\lim_{i\to\infty }q_{i,t_i^\prime }(x_i^\prime)\geq \lambda^2\\ r_{\operatorname{Rm}}(x_\infty^\prime)=\lim_{i\to\infty }r_{\operatorname{Rm}}(x_i^\prime, t_i^\prime)< c\lambda,
\end{align*}
which is a contradiction (see \cite[Claim $15.85$]{Bam3}).
    
\end{proof}

The following lemma is the key new ingredient in the proof of Theorem \ref{thm:epsreg}. We first make the following definition.

\begin{Definition} \label{def: diameter}
    Let $(M^n,(g_t)_{t\in I})$ be a Ricci flow and $Y\subset M\times \{t\}$. The $P^\ast$-parabolic diameter of $Y$, written $
    \operatorname{diam}_{P^\ast}(Y)$, is defined to be the infimum over all $r>0$ such that for each $x,y\in Y$ we have $$d_{W^1}^{g_{t-r^2}}(\nu_{x,t;t-r^2},\nu_{y,t;t-r^2})<r.$$
\end{Definition}

\begin{Lemma} \label{lem: techlemma}
	For any  $Y, \Lambda<\infty$ and $\epsilon>0$, there exists $C=C(Y)<\infty$ such that the following holds if $\delta\leq \overline{\delta}(Y,\Lambda,\epsilon)$. Let $(M^n,(g_t)_{t\in I})$ be a closed Ricci flow, and suppose $(x_0,t_0)\in M\times I$, $r>0$ satisfy $\mathcal{N}_{x_0,t_0}(r^2)\geq -Y$. Assume $(x_0,t_0)$ is $(n-4,\delta,r)$-selfsimilar and $(\delta ,r)$-static,
    $h$ is a strong $(n-4,\epsilon,r)$-soliton potential, $y=(y^1,...,y^{n-4})$ are strong $(n-4,\epsilon,r)$-splitting maps, and $q:=4\tau(h-W)$. For $z\in B(0^{n-4},\Lambda r),$ $s\in [\Lambda^{-1}r,\Lambda r]$, and $t\in [t_0-\Lambda r^2,t_0-\Lambda^{-1}r^2]$, $$\Sigma_{z,s,t}\coloneqq y_t^{-1}(z)\cap q_t^{-1}(-\infty,s^2] \cap P^\ast(x_0,t_0;C \Lambda r)$$  
    satisfy the following:
	\begin{enumerate}
 		\item $\Sigma_{z,s, t}\subset \subset P^\ast(x_0,t_0;C\Lambda  r)$,\label{iden:techlemma2}
		\item $\partial \Sigma_{z,s, t} =q_t^{-1}(s^2)\cap y_t^{-1}(z) \cap P^\ast(x_0,t_0;C\Lambda r)$ is $\epsilon$-close in the $C^{\lfloor \epsilon ^{-1} \rfloor }$-topology to the round $\mathbb{S}^3/\Gamma$ for some $\Gamma \subset O(3,\mathbb{R})$ (here $\partial \Sigma_{z,s,t}$ denotes the boundary when viewed as a subset of $y_t^{-1}(z)$), \label{iden:techlemma3}
        
        \item $\operatorname{diam}_{P^\ast}(\Sigma_{z,s,t})\leq 4s$. \label{iden:techlemma5}
	\end{enumerate}
\end{Lemma}

\begin{Remark} \label{remark:diameter}
    Only \ref{iden:techlemma2}-\ref{iden:techlemma3} will be required in the proof of Theorem \ref{thm:epsreg}, \ref{iden:techlemma5} has been included as we believe it will be relevant for future work on this topic.
\end{Remark}

\begin{proof}
By means of parabolic rescaling and application of a time-shift, we may assume that $r=1,$ and $t_0=0$. Suppose for the sake of contradiction that, for some $C=C(Y)<\infty$ to be determined, at least one of \ref{iden:techlemma2}-\ref{iden:techlemma3} were false. That is, there is a sequence $\delta_i\searrow 0$ together with closed Ricci flows $(M^n_i, (g_{i,t})_{t\in [-\delta_i^{-1},0]})$, $(x_{0,i},0)\in M_i\times \{0\}$ which are strongly $(n-4,\delta_i,1)$-selfsimilar, $(\delta_i,1)$-static, $\mathcal{N}_{x_{0,i},0}(1)\geq -Y$, but at least one of \ref{iden:techlemma2}-\ref{iden:techlemma3} fails for each $i\in\mathbb{N}$. In particular, at least one of \ref{iden:techlemma2}-\ref{iden:techlemma3}  must fail infinitely often. By Corollary \ref{cor : uni_cont},  we can pass to a subsequence so that the $\mathbb{F}$-convergence \eqref{eq:Fconvergence} holds, where $\mathcal{X}$ is a static flow modeled on the Ricci flat cone $C(\mathbb{S}^3/\Gamma)\times \mathbb{R}^{n-4}$ with vertex $x_\infty\coloneqq (x_\ast,0^{n-4})$. Moreover, we have $\psi_i^{\ast} q_i \to q$ and $\psi_i^{\ast} y^j_i \to y^j$ in $C_{\operatorname{loc}}^{\infty}(\mathcal{R})$, where $ y^j$ is the projection onto the $j$-th Euclidean factor and $q$ is the square of the radial coordinate function. Suppose that 
\ref{iden:techlemma2} fails infinitely often, then after passing to a subsequence, for each $i\in \mathbb{N}$ there exists $z_i\in B(0^{n-4};\Lambda ), s_i\in [\Lambda^{-1},\Lambda ], t_i\in [-\Lambda ,-\Lambda^{-1}]$, such that $\Sigma_{z_i,s_i,t_i}$ is not compactly contained in $P^\ast(x_{0,i},0;C \Lambda)$. Then we can choose points $(x_i^\prime,t_i )\in  \Sigma_{z_i,s_i,t_i}$ such that $$d_{W^1}^{g_{t_i-C^2\Lambda^2}}(\nu_{x_{0,i},0;-C^2 \Lambda^2}, \nu_{x_i^\prime,t_i;-C^2\Lambda^2 })=C\Lambda.$$ Moreover, by Lemma \ref{lem:bubbles1}\ref{iden:bubbles11} we have  $|q_{i,t_i}(x_i^\prime)|\leq \Lambda$. By Corollary \ref{cor : uni_cont}, we can pass to a subsequence so that $(x_i^\prime, t_i)\xrightarrow[i\to\infty]{\mathfrak{C}} x_\infty^\prime $ for some $x_\infty^\prime\in \mathcal{X}_{t_\infty}$ where $t_\infty \coloneqq \lim_{i\to\infty} t_i$, with 
$$d_{W^1}^{\mathcal{X}_{-C^2\Lambda^2}}(\nu_{x_\infty^\prime;-C^2 \Lambda^2},\nu_{x_\infty;-C^2 \Lambda^2})=C \Lambda.$$
On the other hand, we get from \cite[Claim $22.9$]{Bam3} that
\begin{align*}
       d_{W^1}^{\mathcal{X}_{-C^2 \Lambda^2}}(\nu_{x_\infty^\prime;-C^2\Lambda^2},\nu_{x_\infty;-C^2\Lambda^2})& \leq  d_{W^1}^{\mathcal{X}_{t_\infty}}(\delta_{x_\infty^\prime},\nu_{x_\infty;t_\infty})\\ &\leq d_{t_\infty}(x_\infty^\prime, x_\infty(t_\infty))+d_{W^1}^{\mathcal{X}_{t_\infty}}(\delta_{x_\infty(t_\infty)},\nu_{x_\infty ;t_\infty})\\ &\leq (q(x_\infty^\prime)+|y|^2(x_\infty^\prime))^{\frac{1}{2}}+C^\prime\Lambda \\&\leq \lim_{i\to\infty}|q_{i,t_i}|^{\frac{1}{2}}(x_i^\prime)+\lim_{i\to\infty} |y_{i,t_i}|(x_i^\prime)+C^\prime\Lambda\\& \leq C^\prime\Lambda,
    \end{align*}
for some constant $C^\prime(Y)<\infty$. This is a contradiction if we choose $C := \frac{1}{2} C^\prime$. Suppose instead that \ref{iden:techlemma2} fails for infinitely many $i\in \mathbb{N}$, and pass to a subsequence so that $z_i \to z\in \overline{B}(0^{n-4},\Lambda)$ and $s_i \to s\in [\Lambda^{-1},\Lambda]$. The fact that \ref{iden:techlemma3} holds for large $i\in\mathbb{N}$ is now a consequence of \ref{iden:techlemma2}, Lemma \ref{lem:bubbles1}\ref{iden:bubbles12}, the smooth convergence on the regular set, and the fact that $z,s$ are regular values of $y,q$, respectively. In conclusion, neither  \ref{iden:techlemma2} nor \ref{iden:techlemma3} can fail infinitely often, which is a contradiction.

Finally, assume that \ref{iden:techlemma5} were false. That is, there are sequences $\delta_i\searrow 0$ together with closed Ricci flows $(M^n_i, (g_{i,t})_{t\in [-\delta_i^{-1},0]})$, $(x_{0,i},0)\in M_i\times \{0\}$ which are strongly $(n-4,\delta_i,1)$-selfsimilar, $(\delta_i,1)$-static, $\mathcal{N}_{x_{0,i},0}(1)\geq -Y$, but for which \ref{iden:techlemma5} fails. As before, we can pass to a subsequence so that \eqref{eq:Fconvergence} holds and $\psi_i^{\ast}q_i \to q$, $\psi_i^\ast y_i^j \to y^j$ in $C_{\operatorname{loc}}^{\infty}(\mathcal{R})$. Then for each $i\in \mathbb{N}$ there exist $z_i\in B(0^{n-4};\Lambda), t_i\in [-\Lambda,-\Lambda^{-1}]$, $s_i \in [\Lambda^{-1},\Lambda]$ such that $$\operatorname{diam}_{P^\ast}(\Sigma_{z_i,s_i,t_i})\geq 4s_i.$$ By definition, we can choose $x_i^\prime,x_i^{\prime\prime }\in\Sigma_{z_i,s_i,t_i}$ such that $$ d_{W^1}^{g_{t_i-16s_i^2}}(\nu_{x_i^\prime,t_i;t_i-16s_i^2},\nu_{x_i^{\prime\prime},t_i;t_i-16s_i^2})\geq 4s_i.$$
We can pass to a further subsequence so that $s_i \to s_{\infty}\in [\Lambda^{-1},\Lambda]$ and (by Corollary \ref{cor : uni_cont}) 
 \begin{align*}
     (x_i^\prime, t_i)\xrightarrow[i\to\infty]{\mathfrak{C}} x_\infty^\prime  \qquad  (x_i^{\prime\prime}, t_i) \xrightarrow[i\to\infty]{\mathfrak{C}} x_\infty^{\prime \prime}
 \end{align*}
for $x_\infty^\prime,x_\infty^{\prime\prime }\in \mathcal{X}_{t_\infty}$ where $t_\infty =\lim_{i\to\infty }t_i$, with
 $$ d_{W^1}^{g_{t_\infty- 16s_\infty^2}}(\nu_{x_\infty^\prime,t_\infty;t_\infty- 16s_\infty^2},\nu_{x_\infty^{\prime\prime},t_\infty;t_\infty- 16s_\infty^2})\geq 4s_\infty.$$  
 Moreover, we get that
\begin{align*}
    d_{W^1}^{\mathcal{X}_{t_\infty - 16s_\infty^2}}(\nu_{x_\infty^\prime,t_\infty;t_\infty-16s_\infty^2},\nu_{x_\infty^{\prime\prime},t_\infty;t_\infty-16s_\infty^2})\leq d_{t_\infty}(x_\infty^\prime, x_\infty^{\prime\prime} )&\leq d_{t_\infty}((x_\ast,z_\infty),x_\infty^\prime )+d_{t_\infty}((x_\ast,z_\infty),x_\infty^{\prime\prime} )\\& \leq 2s_\infty,
\end{align*}
which is a contradiction.
 \end{proof}

%% file: Epsilon_Regularity.tex
\section{Proof of Theorem \ref{thm:epsreg}} \label{section:epsreg}

We first establish a sufficient condition for a Ricci flow to be almost-static. 

\begin{Lemma} \label{lem:conditionforstatic} For any $Y<\infty$ and $\epsilon>0$, the following holds if $\delta \leq \overline{\delta}(Y,\epsilon)$. Let $(M^n,(g_t)_{t\in I})$ be a closed Ricci flow, and suppose $(x_0,t_0)\in M\times I$, $r>0$ satisfy $\mathcal{N}_{x_0,t_0}(r)\geq -Y$. If
$$r^{2-n}\int_{t_0-2\epsilon^{-1}r^2}^{t_0-\frac{1}{2}\epsilon r^2}\int_{P_t^{\ast}(x_0,t_0;\delta^{-1}r)}R^2 dg_t dt < \delta,$$
then $(x_0,t_0)$ is $(\epsilon,r)$-static.
\end{Lemma}
\begin{proof} We may assume $t_0=0$ and $r=1$. Fix $D<\infty$ to be determined. Because $\nu_{x_0,0;t}(M\setminus P^{\ast-}(x_0,0;D)_t) \leq \Psi(D^{-1}|Y,\epsilon)$ for any $t\in [-2\epsilon^{-1},-\frac{1}{2}\epsilon]$, H\"older's inequality gives
\begin{align*} \int_{-2\epsilon^{-1}}^{-\frac{1}{2}\epsilon}\int_M |R|d\nu_{x_0,0;t}dt \leq &C(Y,\epsilon)\int_{-2\epsilon^{-1}}^{-\frac{1}{2}\epsilon} \int_{P^{\ast}(x_0,t_0;D)_t} |R|dg_t dt \\ &\hspace{-20 mm}+ C(Y,\epsilon) \left( \int_{-2\epsilon^{-1}}^{-\frac{1}{2}\epsilon} \int_{M} |R|^2 d\nu_{x_0,0;t} dt\right)^{\frac{1}{2}} \left( \int_{-2\epsilon^{-1}}^{-\frac{1}{2}\epsilon} \nu_{x_0,0;t}(M\setminus P^{\ast}(x_0,t_0;D)_t) \right)^{\frac{1}{2}}\\  \leq &C(Y,\epsilon,D)\delta^{\frac{1}{2}} + \Psi(D^{-1}|Y,\epsilon).
\end{align*}
By choosing $D=D(Y,\epsilon)$ sufficiently large, we thus have
$$\int_{-2\epsilon^{-1}}^{-\frac{1}{2}\epsilon}\int_M |R|d\nu_{x_0,0;t}dt \leq \Psi(\delta|Y,\epsilon),$$
and the claim then follows from the proof of \cite[Claim $22.7$]{Bam3}.
\end{proof}

We now begin the proof of Theorem \ref{thm:epsreg}.

\begin{proof}
    By means of parabolic rescaling and time translation, we may assume that $r=1$ and $t_0=0$. Let $\zeta>0$ be a constant to be determined. Choose a scale $r_0(Y,\zeta)>0$ such that there is some $s\in [r_0^2,r_0]$ for which $(x_0,t_0)$ is strongly $(\zeta,s)$-selfsimilar. By our assumption, we must have $$s^{2-n}\int_{-s^2r_0^{-2}}^{-s^2r^2_0}\int_{P^\ast_t(x_0,t_0;sr_0^{-1})}R^2\, dg_tdt \leq r_0^{4-2n}\int_{-2}^{-\epsilon_0}\int_{P_t^{\ast}(x_0,t_0;1)}|Rm|^2\,dg_tdt<r_0^{4-2n}\epsilon_0$$
    if $\epsilon_0\leq \epsilon_0(Y,\zeta)$. Shrinking $r_0(Y,\zeta)$ if necessary, we get from Lemma \ref{lem:conditionforstatic} that $(x_0,t_0)$ is $(\zeta,s)$-static if $\epsilon_0<\epsilon_0(\zeta,Y)$. Moreover, if $\epsilon\leq \epsilon_0(Y,\zeta)$, we conclude that $(x_0,t_0)$ is strongly $(n-4,\zeta,s)$-split. By parabolic rescaling and Proposition \ref{prop: strong_potentialsl}, we may assume that $(x_0,0)$ is $(\zeta,1)$-static, and strongly $(n-4,\zeta,1)$-selfsimilar. Set $q=4\tau(h-W)$ where $W\coloneqq \mathcal{N}_{x_0,0}(1)$ and $h$ is a strong $(n-4,\zeta,1)$-soliton potential, and let $y=(y^1,...,y^{n-4})$ denote strong $(n-4,\zeta,1)$-splitting maps. Let $\Sigma_{z,\lambda ,t}$ be as in the statement of Lemma \ref{lem: techlemma}. Assume for the sake of contradiction that the result were false. Then there is a sequence $\zeta_i\searrow 0$ together with closed Ricci flows $(M_i,(g_{i,t})_{t\in [-\zeta_i^{-1},0]})$ with $(x_{0,i},0)\in M_i\times \{0\}$ that are $(\zeta_i,1)$-static, strongly $(n-4,\zeta_i,1)$-selfsimilar, $r_{\operatorname{Rm}}(x_{0,i},0) \leq \zeta_i$ and for which $$\int_{-2}^{-1}\int_{P^{\ast}_t(x_{0,i},0;\zeta_i^{-1})}|Rm_{i,t}|_{g_{i,t}}^2\,dg_{i,t}dt<\zeta_i.$$ By Theorem \ref{bamconvergence}, we have $\mathbb{F}$-convergence \eqref{eq:Fconvergence} where $\mathcal{X}$ is a static flow modeled on the  flat cone $C(\mathbb{S}^3/\Gamma)\times \mathbb{R}^{n-4}$ with vertex $x_\infty\coloneqq (x_{\ast},0^{n-4})$.
   \begin{Claim} \label{claim:epsreg1}
    There exists $t_i\in [-2,-1]$ and $z_i\in B(0^{n-4},1)$ for which $\Sigma_i :=\Sigma_{z_i,1,t_i}$ is a smooth submanifold and  
    $$ \int_{\Sigma_i}|Rm_{i,t}|_{g_{i,t}}^2\, d\mathcal{H}_{g_{i,t}}^{n-4}<2\zeta_i.$$
    \end{Claim}
    \begin{proof}
            By Corollary \ref{cor : uni_cont}, Lemma \ref{lem: techlemma}, and the coarea formula we have 
        \begin{align*}
       &  \int_{-2}^{-1} \int_{B(0^{n-4};1)}\int_{\Sigma_{z,1,t}}|Rm_{i,t}|_{g_{i,t}}^2\,d\mathcal{H}_{g_{i,t}}^{n-4}dzdt  \\ &\leq 2\int_{-2}^{-1} \int_{B(0^{n-4};1)}\int_{\Sigma_{z,1,t}}\frac{1}{|\nabla y^{i}_{1}\wedge...\wedge \nabla y^{i}_{n-4}|_{g_{i,t}}}\cdot|Rm_{i,t}|_{g_{i,t}}^2\,d\mathcal{H}_{g_{i,t}}^{n-4}dzdt\\&\leq 2\int_{-2}^{-1}\int_{P^{\ast}_t(x_{0,i},0;\zeta_i^{-1})}|Rm_{i,t}|_{g_{i,t}}^2\,dg_{i,t}dt \\ & \leq 2\zeta_i.
 \end{align*}
By Sard's Theorem, we can choose $(z,t)\in B(0^{n-4},1) \times [-2,-1]$ such that $\Sigma_{z,1,t}$ is smooth and the desired estimate holds.
\end{proof}

Let $\Sigma_i$ be a sequence of submanifolds as in Claim \ref{claim:epsreg1}. Recall the differential character $\widehat{c}_2$ from Section \ref{section:ChernSimons}. We have $\psi_i^{\ast}g_i \to g$, $\psi_i^{\ast}q_i \to q$, and $\psi_i^{\ast}J_i \to J$, where $(g,J)$ is the standard K\"ahler structure on $\mathcal{C}:=\mathbb{C}^{n-2}\times \mathbb{C}^2/\Gamma$, and $q$ is the distance squared to the origin on the $C(\mathbb{S}^3/\Gamma)$ factor. For any K\"ahler manifold $(M',J')$, we let $T_{J'}^{1,0}M' \subseteq TM'\otimes_{\mathbb{R}}\mathbb{C}$ denote the the tangent vectors satisfying $J'v=\sqrt{-1}v$. Although $d(\psi_i^{-1})$ does not map $T_{J_i}^{1,0}M_i$ into $T_{J}^{1,0}\mathcal{C}_i$, if we let $\Pi_i : T_{\psi_i^{\ast}J_i}^{1,0}U_i \to T_{J}^{1,0}U_i$ be the projection map, then $\Pi_i$ converges locally smoothly to the identity map, and $\Psi_i := \Pi_i\circ d(\psi_i^{-1})$ are complex bundle isomorphisms $(T^{1,0}_{J_i}M_i,J_i)\to (T_J^{1,0}\mathcal{C},J)$ lying over $\psi_i^{-1}$. By Proposition \ref{prop:secondchernproperties}\ref{secondchern1}, 
$$\langle\widehat{c}_2(T^{1,0}_{J_i}M_i, \nabla^{g_{i,t}}),\partial \Sigma_i\rangle = \langle \widehat{c}_2(T_J^{1,0}\mathcal{C},\Psi_i^{\ast}\nabla^{g_{i,t}}),\psi_{i,t}^{-1}(\partial \Sigma_i)\rangle.$$
Because $\psi_i^{\ast}q_i\to q$, $\psi_i^{\ast}g_i \to g$ in $C_{\operatorname{loc}}^{\infty}$, it follows that $\psi_i^{-1}(\Sigma_i)\to (\mathbb{S}^3/\Gamma) \times \{z\}$ smoothly. By 
the locally smooth convergence $\psi_{i,t}^{\ast}g_{i,t} \to g_t$ and $\psi_{i,t}^{\ast}J_i \to J$, we also have locally smooth convergence of the pullback $\Psi_i^{\ast}\nabla^{g_{i,t}}$ of the Chern connection to the Chern connection $\nabla^g$ of $(T_J^{1,0}\mathcal{C},g)$. By
Proposition \ref{prop:secondchernproperties}\ref{secondchern3}, it follows that
$$\lim_{i\to \infty}\langle \widehat{c}_2(T_J^{1,0}\mathcal{C},\Psi_i^{\ast}\nabla^{g_{i,t}}) -\widehat{c}_2(T_J^{1,0}\mathcal{C},\nabla^g),\psi_{i,t}^{-1}(\partial \Sigma_i)\rangle \equiv 0 \mod \mathbb{Z}.$$
Because $\nabla^g$ is a flat connection and $\psi_{i,t}^{-1}(\partial \Sigma_i)$ is homologous to $\mathbb{S}^3/\Gamma \times \{0\}$, we know by Proposition \ref{prop:secondchernproperties}\ref{secondchern2} that  
$$\langle\widehat{c}_1(T_J^{1,0}\mathcal{C},\nabla^g),\psi_{i,t}^{-1}(\partial \Sigma_i)\rangle \equiv \langle \widehat{c}_2(T_J^{1,0}\mathcal{C},\nabla^g),\{z\}\times \mathbb{S}^3/\Gamma\rangle \equiv \frac{1}{|\Gamma|} \mod \mathbb{Z}.$$
Combining expressions, we thus have
$$\lim_{i\to \infty} \langle \widehat{c}_2(T_{J_i}^{1,0}M_i,\nabla^{g_{i,t}}),\partial \Sigma_i\rangle \equiv \frac{1}{|\Gamma|} \mod \mathbb{Z}.$$ 
On the other hand, Proposition \ref{prop:secondchernproperties}\ref{secondchern2}
and Claim \ref{claim:epsreg1} give
$$\langle \widehat{c}_2(T_{J_i}^{1,0}M_i,\nabla^{g_{i,t}}), \partial \Sigma_i \rangle \equiv  \frac{1}{8\pi^2}\int_{\Sigma_i} \left( \operatorname{tr}(F_{\nabla^{g_{i,t}}})^2 -\operatorname{tr}(F_{\nabla^{g_{i,t}}}^2) \right)\xrightarrow[i\to \infty]{}0,$$ 
so that $\frac{1}{|\Gamma|} \equiv 0 \mod \mathbb{Z}$. That is,  $\Gamma$ is trivial, and $\mathcal{C} = \mathbb{C}^n$, contradicting $r_{\operatorname{Rm}}(x_{i,0},0)\leq \zeta_i$.

Finally, we drop the K\"ahler assumption, but suppose by way of contradiction that $n=4$. Because $\psi_i^{-1}(\partial \Sigma_i)$ converges smoothly to $\mathbb{S}^3/\Gamma$ and $\psi_i^{\ast}g_i \to g$ in $C_{\operatorname{loc}}^{\infty}$, it follows that the principal curvatures of the embeddings $\partial \Sigma_i \hookrightarrow M_i$ all converge to 1, and 
$$\lim_{i\to \infty} \mathcal{H}_{g_i}^3(\partial \Sigma_i) = \mathcal{H}_g^3(\mathbb{S}^3/\Gamma) = \frac{2\pi^2}{|\Gamma|}.$$
These observations, along with Claim \ref{claim:epsreg1}, 
combine with the Chern-Gauss-Bonnet formula \ref{eq:cherngaussbonnet} to give
$$32\pi^2 \chi(\partial \Sigma_i) = 16 \cdot \frac{2\pi^2}{|\Gamma|}+\Psi(i^{-1}).$$
Because $\chi(\partial \Sigma_i) \in \mathbb{Z}$, this implies $|\Gamma|=1$ for sufficiently large $i\in \mathbb{N}$, which is again a contradiction.
\end{proof}

\begin{proof}[Proof of Remark \ref{rem:riemannian}]
By arguing as in the proof of Theorem \ref{thm:epsreg}, using Proposition \ref{prop:firstpontryaginproperties} instead of \ref{prop:secondchernproperties} we get
$$\lim_{i\to \infty} \langle \widehat{p}_1(TM_i,\nabla^{g_{i,t}}),\partial \Sigma_i\rangle \equiv \langle \widehat{p}_1(T(C(\mathbb{S}^3/\Gamma)\times \mathbb{R}^{n-4}),\nabla^g\rangle \mod \mathbb{Z},$$
whereas by Proposition \ref{prop:firstpontryaginproperties}\ref{firstpontryagin1}, we have
$$\langle \widehat{p}_1(TM_i,\nabla^{g_{i,t}}),\partial \Sigma_i\rangle \equiv -\frac{1}{8\pi^2} \int_{\partial \Sigma_i}\operatorname{tr}(F_{\nabla^{g_{i,t}}}^2) \xrightarrow[]{i\to \infty} 0 \mod \mathbb{Z}.$$
By Proposition \ref{prop:firstpontryaginproperties}\ref{firstpontryagin4}, this yields a contradiction unless $\mathbb{S}^3/\Gamma$ is an exceptional lens space. 
\end{proof}

%% file: Minkowski2.tex
\section{Minkowski Estimates for the Singular Set} \label{section : Minkowski_estimates}

We will need the following lemmas.

\begin{Lemma} \label{lem:non_collapsing}
For every $Y<\infty$, there exists a constant $C(Y)<\infty$ such that the following holds.    Let $(M^n,(g_t)_{t\in I})$ be a Ricci flow on a closed manifold, and suppose $(x_{0},t_{0})\in M\times I$, $r>0$ satisfy $[t_0-2r^2,t_0]\subseteq I$ and $\mathcal{N}_{(x_{0},t_{0})}(r^{2})\geq -Y$. Then for all $s\in [0,r]$ we have 
    \begin{align*}
      C(Y)^{-1}s^{n+2}\leq  |P^{\ast}(x_{0},t_{0};s)|\leq C(Y)s^{n+2}.
    \end{align*}
\end{Lemma}
\begin{proof}
    By means of parabolic rescaling, we may assume that $r=1$ and $t_0=0$. By the monotonicity of $\mathcal{N}_{x_0,t_0}$, we can also assume $s=1$. The upper bound is an immediate consequence of \cite[Theorem 9.8]{Bam1}. To see the lower bound, we will follow the argument of \cite[Claim 17.53(e)]{Bam3}.
    
   \begin{Claim} \label{claim:non_collapsing}
       There exists $c>0$ such that for all $t\in [-c^{2},0)$, if $(x,t)$ is an $H_{n}-$center of $(x_{0},0)$, then $B(x,t;c)\times \{t\}\subset P^{*}(x_{0},0;1)$.
   \end{Claim} 
    
   \begin{proof}
       Set $c:=\frac{1}{1+H_n}$, and let $(y,t)\in B(x,t;c)\times\{t\}$. Then
    \begin{align*}
        d_{W_{1}}^{g_{-1}}(\nu_{y,t;-1},\nu_{x_{0},0;-1})\leq d_{W_{1}}^{g_t}(\delta_{y},\nu_{x_{0},0,t})\leq d_{W_{1}}^{g_t}(\delta_{y},\delta_{x})+d_{W_{1}}^{g_t}(\delta_{x},\nu_{x_{0},0;t})\leq c+\sqrt{H_{n}}c<1.
    \end{align*}
    \end{proof}
    By \cite[Theorem 6.2]{Bam1}, there exists $c^\prime(Y)>0$ such that for all $t\in [-\frac{c^{2}}{2H_{n}},-\frac{c^{2}}{4H_{n}}]$, if $(x,t)$ is a $H_{n}$-center of $(x_0,0)$ then 
    \begin{align*}
        |B(x,t;c)|_{g_t}\geq |B(x,t;\sqrt{2H_{n}|t|})|_{g_t}\geq c^{\prime}(Y).
    \end{align*}
    By Claim \ref{claim:non_collapsing} we get that
    $$|P^{\ast}(x_0,0;1)|\geq \int_{-\frac{c^2}{2H_n}}^{-\frac{c^2}{4H_n}} |B(x,t;c)|_{g_t}\, dt \geq c(Y),$$
  which gives the result.
\end{proof}

We define the effective strata as follows (c.f. \cite[Definition $11.1$]{Bam3}).

\begin{Definition}
    For $\epsilon>0$ and $ 0<r_1<r_2$, the effective strata $$\widetilde{\mathcal{S}}^{0,\epsilon}_{r_1,r_2}\subset \widetilde{\mathcal{S}}^{1,\epsilon}_{r_1,r_2}\subset ...\subset \widetilde{\mathcal{S}}^{n-2,\epsilon}_{r_1,r_2}$$ are defined as follows: $(x,t)\in \widetilde{\mathcal{S}}^{k,\epsilon}_{r_1,r_2}$ if and only if there does not exist $r\in (r_1,r_2)$ satisfying one of the following: 
 \begin{enumerate}
        \item $(x_0,t_0)$ is strongly $(k+1,\epsilon, r)$-selfsimilar,
        \item $(x_0,t_0)$ is strongly $(k-1,\epsilon,r)$-selfsimilar and $(\epsilon,r)$-static.
    \end{enumerate} 
\end{Definition}

We will need the following consequence of of \cite[Proposition $11.2$]{Bam3}.

\begin{Lemma} \label{lem:quant_strat}
   For any $Y<\infty$ and $\epsilon>0$, there exists $C(Y,\epsilon)<\infty$ such that the following statement holds. Let $(M^{n},g(t)_{t\in I})$ be a Ricci flow on a closed manifold, and suppose $(x_{0},t_{0})\in M\times I$, $r_2 \geq r_1>0$ satisfy $\mathcal{N}_{x_{0},t_{0}}(r_2^{2})\geq -Y$. Then
\begin{align*}
        |\widetilde{\mathcal{S}}^{k,\epsilon}_{r_{1},r_{2}}\cap P^{*}(x_{0},t_{0};r_{2})|\leq C \left(\frac{r_{1}}{r_{2}} \right)^{n+2-k-\epsilon}r_2^{n+2}.
    \end{align*} 
\end{Lemma}

\begin{proof}
    By means of parabolic rescaling, we can assume that  $r_2 = 1$. It is a consequence  of \cite[Proposition $11.2$]{Bam3} that
    there exist $(x_1,t_1),...,(x_N,t_N) \in \widetilde{S}_{r_1,1}^{\epsilon,k} \cap P^{\ast}(x_0,t_0;1)$ such that $N\leq C(Y,\epsilon)r_1^{-k-\epsilon}$ and
    \begin{equation} \label{eq:rescaledquantstrat} \widetilde{\mathcal{S}}_{r_1,1}^{k,\epsilon} \cap P^{\ast}(x_0,t_0;1) \subseteq \bigcup_{i=1}^N P^{\ast}(x_i,t_i;r_1).\end{equation}
    In fact, \cite[Proposition 11.2]{Bam3} applies to a slightly different notion of effective strata, but the two are equivalent (up to alteration of some constants) by \cite[Proposition 12.1]{Bam3}, \cite[Proposition 3.2]{HJ}, and Proposition \ref{prop: strong_potentialsl}. Combining \eqref{eq:rescaledquantstrat} and Lemma \ref{lem:non_collapsing} yields    
    \begin{align*}
         |\widetilde{\mathcal{S}}^{k,\epsilon}_{r_{1},1}\cap P^{*}(x_{0},t_{0};1)|\leq C(Y,\epsilon) r_{1}^{n+2-k-\epsilon}.
    \end{align*}
    \end{proof}

Now we begin the proof of Theorem \ref{thm:smoothminkowski}.

\begin{proof}
    By means of parabolic rescaling and application of a time shift, we may assume that $r=1$ and $t_{0}=0$. By \cite[Proposition 5.2]{Bam1} and \cite[Corollary 5.11]{Bam1}, there exists $Y'=Y'(Y,A)<\infty$ such that for all $(x,t) \in P^{\ast}(x_0,0;A)$, we have $\mathcal{N}_{x,t}(1) \geq -Y'$.

    Fix $\overline{\epsilon}\in (0,1)$ to be determined throughout the proof. Given $\eta,r^{\prime}\in (0,1)$, let $\mathcal{D}_{\eta, r^{\prime}}$ denote the set of points $(x,t)\in P^{*}(x_{0},0;A)$ with $t\in [-A,-\zeta]$ such that 
    \begin{align*}
        \frac{(r^{\prime})^{4}}{|P^{*}(x,t;r^{\prime})|}\int _{t-2(r^{\prime})^{2}}^{t}\int_{P^{*}_{s}(x,t;r^{\prime})}|Rm|_{g_s}^{2}\, dg_sds\geq \eta.
    \end{align*}
    \begin{Claim} \label{claim:smoothminkowski1}
       For every $\eta\in(0,1)$, there exists $C=C(A,Y,\eta)<\infty$ such that $$|\mathcal{D}_{\eta,r^{\prime}}|\leq CD(r^{\prime})^{4}$$
       for every $r^{\prime}\in (0,1)$.
    \end{Claim}
    \begin{proof}
        Choose a cover $(x_{\alpha},t_{\alpha})\in \mathcal{D}_{\eta,r^{\prime}}$ such that $\bigcup_{\alpha} P^{*}(x_{\alpha},t_{\alpha};5r^{\prime})\supseteq \mathcal{D}_{\eta,r^{\prime}}$ and $P^{*}(x_{\alpha},t_{\alpha};r^{\prime})$ are pairwise disjoint. It follows from Lemma \ref{lem:non_collapsing} and Definition \ref{def:finite_energy} that
        \begin{align*}
            |\mathcal{D}_{\eta,r^{\prime}}|\leq \sum_{\alpha }|P^{*}(x_{\alpha},t_{\alpha};5r^{\prime})|& \leq C(A,Y)\sum_{\alpha }|P^{*}(x_{\alpha},t_{\alpha};r^{\prime})|\\ & \leq C(A,Y,\eta)(r^{\prime})^{4}\sum_{\alpha} \int _{t-2(r^{\prime})^{2}}^{t}\int_{P^{*}_{s}(x_{\alpha},t_{\alpha};r^{\prime})}|Rm|_{g_s}^{2}\, dg_sds\\ & \leq C(A,Y,\eta)\cdot D\cdot (r^{\prime})^{4}.
        \end{align*}
    \end{proof}
    Let $m\in \mathbb{N}$ be such that $\sigma\in [2^{-(m+1)},2^{-m}]$. 
    \begin{Claim} \label{claim:smoothminkowski2}
        There exists $C=C(A,Y,\overline{\epsilon})<\infty$  such that for each $i\in \{0,1,...,m-1\}$ we have 
        \begin{align*}
            |\widetilde{\mathcal{S}}^{n-3,\overline{\epsilon}}_{(\sigma,2^{-i})}\cap \mathcal{D}_{\overline{\epsilon},2^{-i}}|\leq C D\sigma^{5-\overline{\epsilon}}(2^{-i})^{\overline{\epsilon}-1}.
        \end{align*}
    \end{Claim}
    \begin{proof}
        Choose a Vitali cover of $\mathcal{D}_{\overline{\epsilon},2^{-i}}$ by parabolic balls of radius $2^{-i}$, as in the proof of Claim \ref{claim:smoothminkowski1}. By Lemma \ref{lem:non_collapsing} and Claim \ref{claim:smoothminkowski1}, there must be at most $C(A,Y)D(2^{-i})^{4-(n+2)}$ balls in this cover. Using Lemma \ref{lem:quant_strat} we get that 
        \begin{align*}
            |\widetilde{\mathcal{S}}^{n-3,\overline{\epsilon}}_{(\sigma,2^{-i})}\cap \mathcal{D}_{(\overline{\epsilon},2^{-i})}|& \leq C(A,Y,\overline{\epsilon})D(2^{-i})^{4-(n+2)}\sigma^{5}(2^{-i})^{n-3}(\frac{2^{-i}}{\sigma})^{\overline{\epsilon}} \leq C(A,Y,\overline{\epsilon})D(\frac{\sigma}{2^{-i}})^{1-\overline{\epsilon}}\sigma^{4}.
        \end{align*}
    \end{proof}
    Summing Claim \ref{claim:smoothminkowski2} over $i$ gives that 
    \begin{align} \label{eq:minkowski1}
        \sum_{i=0}^{m}  |\widetilde{\mathcal{S}}^{n-3,\overline{\epsilon}}_{(\sigma,2^{-i})}\cap \mathcal{D}_{\overline{\epsilon},2^{-i}}|\leq C(A,Y,\overline{\epsilon})D\sigma^{4}.
    \end{align}
    By combining Claim \ref{claim:smoothminkowski1}, Lemma \ref{lem:quant_strat}, and (\ref{eq:minkowski1}), the proof is complete modulo the following claim.
    \begin{Claim} \label{claim:smoothminkowski3}
    There exists a constant $\overline{\epsilon}(Y,A)>0$ such that 
        \begin{align*}
            &\{(x,t)\in M \times I \,\arrowvert \,r_{\operatorname{Rm}}(x,t)<\overline{\epsilon}\sigma \}\cap P^\ast(x_0,0;A)\cap (M\times [-A,-\zeta])\\ &\subseteq \,\bigcup_{i=1}^{m}\,(\widetilde{\mathcal{S}}^{n-3,\overline{\epsilon}}_{(\sigma,2^{-i})}\cap \mathcal{D}_{(\overline{\epsilon},2^{-i})})\cup\mathcal{D}_{\overline{\epsilon},2^{-(m+1)}}\cup \left(\widetilde{\mathcal{S}}^{n-3,\overline{\epsilon}}_{(\sigma,1)} \cap P^\ast(x_0,0;A)\cap (M\times [-A,-\zeta])\right).
        \end{align*}
    \end{Claim}
    \begin{proof}
        It suffices to show that 
        \begin{align*}
                 \{(x,t)\in M \times I \,\arrowvert \,r_{\operatorname{Rm}}(x,t)\geq \overline{\epsilon} \sigma \}&\supseteq \bigcap_{i=1}^{m}\,((\widetilde{\mathcal{S}}^{n-3,\overline{\epsilon}}_{(\sigma,2^{-i})})^{c}\cup \mathcal{D}_{(\overline{\epsilon},2^{-i})}^{c})\,\cap  \mathcal{D}_{\overline{\epsilon},2^{-(m+1)}}^{c} \cap  (\widetilde{\mathcal{S}}^{n-3,\overline{\epsilon}}_{(\sigma,1)})^{c}\\ &=\bigcup_{I\subset \{1,...,m\}} \,(\,\bigcap_{i\in I} (\widetilde{\mathcal{S}}^{n-3,\overline{\epsilon}}_{(\sigma,2^{-i})})^{c}\, \cap\bigcap _{i\in I^{c}}\mathcal{D}_{\overline{\epsilon},2^{-i}}^{c}\cap \mathcal{D}_{\overline{\epsilon},2^{-(m+1)}}^c\cap (\widetilde{\mathcal{S}}^{n-3,\overline{\epsilon}}_{(\sigma,1)})^{c} ).      
        \end{align*}
Fix $I\subseteq \{1,...,m\}$ and take $$(x,t)\in \,\bigcap_{i\in I} (\widetilde{\mathcal{S}}^{n-3,\overline{\epsilon}}_{(\sigma,2^{-i})})^{c}\, \cap\bigcap _{i\in I^{c}}\mathcal{D}_{\overline{\epsilon},2^{-i}}^{c}\cap \mathcal{D}_{\overline{\epsilon},2^{-(m+1)}}\cap (\widetilde{\mathcal{S}}^{n-3,\overline{\epsilon}}_{(\sigma,1)})^{c}.$$ Suppose first that $I\neq \{1,...m\}$, so that we can choose $k\in I \cup \{0\}$ such that $k+1\notin I$. It follows that 
$$(x,t)\in (\widetilde{\mathcal{S}}^{n-3,\overline{\epsilon}}_{(\sigma,2^{-k} )} )^c\cap \widetilde{\mathcal{S}}^{n-3,\overline{\epsilon}}_{(\sigma,2^{-(k+1)} )}\cap \mathcal{D}_{\overline{\epsilon},2^{-(k+1)}}^{c}, $$ so that there exists $r_{(x,t)}\in (2^{-(k+1)},2^{-k})$ such that $(x,t)$ is either strongly $(n-2,\overline{\epsilon},r_{(x,t)})$-selfsimilar, or both $(\overline{\epsilon},r_{(x,t)})$-static and strongly $(n-4,\overline{\epsilon},r_{(x,t)})$-selfsimilar, and 
  \begin{align*}
         \frac{(2^{-(k+1)})^{4}}{|P^{*}(x,t;2^{-(k+1)})|}\int _{t-2\cdot 2^{-2(k+1)}}^{t}\int_{P^{*}_{s}(x,t;2^{-(k+1)})}|Rm|_{g_s}^{2}\, dg_s ds<\overline{\epsilon}.
   \end{align*}
Suppose first that the former holds. By the proof of Theorem \ref{thm:epsreg},   we may assume that $(x,t)$ is $(\overline{\epsilon},r_{(x,t)})$-static if $\overline{\epsilon}\leq \overline{\epsilon}_1(Y,A) $. It follows from \cite[Proposition $17.1$]{Bam3} that $r_{\operatorname{Rm}}(x,t)\geq \overline{\epsilon}r_{(x,t)}\geq \overline{\epsilon}\sigma $ if $\overline{\epsilon}\leq \overline{\epsilon}_1(Y,A) $.  In the latter case, we get from Theorem \ref{thm:epsreg} that $r_{\operatorname{Rm}}(x,t)\geq \overline{\epsilon}r_{(x,t)}\geq \overline{\epsilon}\sigma$ if $\overline{\epsilon}\leq \overline{\epsilon}_2(Y,A)$. The result then follows by taking $\overline{\epsilon}\leq \min\{\overline{\epsilon}_1,\overline{\epsilon}_2\}$. If $I=\{1,...,m\}$, then $$ (x,t)\in (\widetilde{\mathcal{S}}^{n-3,\overline{\epsilon}}_{(\sigma,2^{-m} )} )^c\cap  \mathcal{D}_{\overline{\epsilon},2^{-(m+1)}}^{c},$$ and the result follows in the same way.

\end{proof}
\end{proof}

\begin{proof}[Proof of Corollary \ref{cor : limitminkowski}]
This follows from Theorem \ref{thm:smoothminkowski}, \cite[Lemma $15.16$, Lemma $15.22$]{Bam3}, and a limiting argument.
\end{proof}

%% file: FanoKRF.tex
Now suppose $(M,J,(g_t)_{t\in [0,T)})$ is a compact Fano manifold of real dimension $n$, and set $m:=\frac{n}{2}$. Let $\omega_t =g_t(J\cdot,\cdot)$ denote the K\"ahler forms, and suppose $\omega_0 \in 2\pi c_1(M)$. Then the Gromov-Hausdorff limit 
$$(X,d):= \lim_{t\to T} (M,(T-t)^{-\frac{1}{2}}d_{g_t})$$
exists, is uniquely determined by the underlying complex manifold $(M,J)$, and has the structure of a singular K\"ahler-Ricci soliton \cite{BamScal,CW1,CSW,HanLi}. 
Moreover, if $(\nu_{x_0,T;t})_{t\in [0,T)})$ denotes any conjugate heat kernel based at the singular time, and $g_{r,t}:=r^{-2}g_{T+r^2 t}$, $\nu_{t}^r:= \nu_{x_0,T;T+r^2 t}$, then Theorem \ref{bamconvergence} gives
$$(M,(g_{r,t})_{t\in [-Tr^{-2},0)},(\nu_{t}^r)_{t\in [-Tr^{-2},0)}) \xrightarrow[r\to 0]{\mathbb{F}} (\mathcal{X},(\mu_t)_{t\in (-\infty,0)}),$$
where $\mathcal{X}$ is a metric soliton modeled on  $(X,d)$.

\begin{proof}[Proof of Theorem \ref{thm:Kahler}] By \cite[Theorem 1.1]{perlKahler}, there exists $C<\infty$ such that 
$$\sup_{t\in [0,T)} (T-t)|R_{g_t}| + (T-t)^{-1}\operatorname{diam}_{g_t}(M) \leq C,$$
hence for any $r\in (0,1]$ and $\epsilon>0$, we have
$$\sup_{t\in [-Tr^{-2},-\epsilon]} |R_{g_{r,t}}|\leq Cr^2 \sup_{t\in [0,T-\epsilon r^2]} |R_{g_{r,t}}| \leq C\epsilon^{-2}.$$
Recalling that $[\omega_t]=2\pi(T-t)c_1(M)$, we have $[\omega_{r,t}]=2\pi |t| c_1(M),$ so that for any $t\in [-\epsilon^{-1},-\epsilon]$
$$\int_{M} R_{g_{r,t}}^2 \omega_{r,t}^m \leq C\epsilon^{-2}\int_M R_{g_{r,t}}\omega_{r,t}^m \leq C(\epsilon) \langle c_1(M)^m,[M]\rangle,$$
where $H^{2n}(M,\mathbb{Z})$ denotes the fundamental class of $M$. 
By \cite[Corollary 4.8]{extremal}, we have
$$\int_M (R_{g_{r,t}}^2-|Rc|_{g_{r,t}}^2)\omega_{r,t}^m = 4m(m-1) \pi^2 \langle c_1(M)^2 \cup [\omega_{r,t}]^{m-2},[M]\rangle$$
$$\int_M (|Rc|_{g_{r,t}}^2-|Rm|_{g_{r,t}}^2)\omega_{r,t}^m = m(m-1)\left \langle \left(4\pi^2 c_1(M)^2-8\pi^2c_2(M) \right) \cup [\omega_{r,t}]^{m-2},[M]\right\rangle,$$
where $\cup$ denotes the cup product of cohomology classes. Combining expressions yields
$$\sup_{t\in [-\epsilon^{-1},-\epsilon]}\int_{M} |Rm|_{g_{s,t}}^2 \omega_{s,t}^m \leq C(\epsilon).$$
By a diagonal argument, the rescaled flows $(\omega_{s,t})$ satisfy the $\Phi$-finite energy condition based at $(x_0,t_i)$ for any $x_0 \in M$, and for some sequence $t_i \nearrow 0$. From Corollary \ref{cor : limitminkowski} and the fact that $X$ is compact (by the diameter bound along the flow), we can cover
\begin{align*}
    \{ y \in \mathcal{X} \: | \: r_{\operatorname{Rm}}'(y)<\sigma\} \cap \mathfrak{t}^{-1}([-2,-\frac{1}{2}])
\end{align*}
can be covered by at most $C\sigma^{-n+2}$ $P^{\ast}$-parabolic balls of radius $\sigma$. This implies
\begin{align*} \int_{-2}^{-\frac{1}{2}} | \{r_{\operatorname{Rm}}^{(\mathcal{X}_t,d_t)}<\sigma \} | dt \leq C \sigma^4.
\end{align*}
Given this, the claim follows by arguing as in \cite[Proposition 3.11]{CHM}.
\end{proof}